\documentclass[preprint,11pt]{elsarticle}

\usepackage{color}
\usepackage{mathtools}
\usepackage{hyperref} 
\usepackage{setspace}
\usepackage{amsfonts}
\usepackage{amssymb}
\usepackage{amsthm}
\usepackage{bm}
\usepackage{ulem}
\usepackage{caption}
\usepackage{subcaption}
\usepackage{multirow}

\biboptions{numbers,compress}

\numberwithin{equation}{section}

\newtheorem{example}{Example}[section]

\newtheorem{theorem}[example]{Theorem}
 
\newtheorem{lemma}[example]{Lemma}

\newtheorem*{maintheorem*}{Main Theorem}
\allowdisplaybreaks
\numberwithin{equation}{section}

\renewcommand{\i}{\ifmmode\mathit{\mathchar"7010 }\else\char"10 \fi}
\renewcommand{\j}{\ifmmode\mathit{\mathchar"7011 }\else\char"11 \fi}
\newcommand{\R}{\mathbb{R}}

\newcommand{\norm}[1]{\left\|#1\right\|}

\newcommand{\pt}{\partial_t}

\newcommand{\ptt}{\partial_{tt}^2}

\newcommand{\px}{\partial_x }

\def\begi{\begin{itemize}}
\def\endi{\end{itemize}}
\def\bega{\begin{array}}
\def\enda{\end{array}}

\def\x{\mathbf{x}}

\def\R{\mathbb{R}}
\def\Xi{{\bm\xi}}

{%

\begin{enumerate}}%
{\end{enumerate}}

{%

\begin{enumerate}}%
{\end{enumerate}}

\journal{}
\begin{document}

\begin{frontmatter}



\title{A spectral method for dispersive solutions of the nonlocal Sine-Gordon equation}


\author[dei]{A. Coclite}
\ead{alessandro.coclite@poliba.it}

\author[mat]{L. Lopez\corref{cor}}
\ead{luciano.lopez@uniba.it}

\author[dei]{S. F. Pellegrino}
\ead{sabrinafrancesca.pellegrino@poliba.it}

\cortext[cor]{Corresponding author}

\address[dei]{Dipartimento di Ingegneria Elettrica e dell'Informazione (DEI), Politecnico di Bari,\\ Via Re David 200 -- 70125 Bari, Italy}

\address[mat]{Dipartimento di Matematica, Universit\`{a} degli Studi di Bari,\\ Via Orabona 4 -- 70125 Bari, Italy}

\begin{abstract}
Moved by the need for rigorous and reliable numerical tools for the analysis of peridynamic materials, the authors propose a model able to capture the dispersive features of nonlocal soliton-like solutions obtained by a peridynamic formulation of the Sine-Gordon equation. The analysis of the Cauchy problem associated to the peridynamic Sine-Gordon equation with local Neumann boundary condition is performed in this work through a spectral method on Chebyshev polynomials nodes joined with the St{\"o}rmer-Verlet scheme for the time evolution. The choice for using the spectral method resides in the resulting reachable numerical accuracy, while, indeed, Chebyshev polynomials allow straightforward implementation of local boundary conditions. Several numerical experiments are proposed for thoroughly describe the ability of such scheme. Specifically, dispersive effects of the specific peridynamic kernel are demonstrated, while the internal energy behavior of the specified peridynamic operator is studied.

\end{abstract}
\end{frontmatter}


\section*{Introduction}

The Sine-Gordon Equation is a nonlinear partial differential equation that describes one-dimensional waves in a continuous media. It is often seen in various field theories, such as condensed matter physics and nonlinear optics. In the one-dimensional case, it is an example of a completely integrable system, meaning that the equation possesses an infinite number of conservation laws and that its solutions can be written in terms of spectral parameters through the inverse scattering method~\cite{asano1990algebraic,liang2021infinitely,wang2021novel,blas2020riccati}. This property leads to exact solutions of the Sine-Gordon equation, known as solitons. Solitons are stable and localized nonlinear waves that retain their shape as they propagate. They have a crucial importance both in mathematics and physics since many physical systems present soliton-like behaviors. Solitons are widely seen in the theory but also exist as sound waves in nonlinear materials and as kinks or domain walls in nano-magnetic materials. The Sine-Gordon equation reads:

\begin{equation}
\label{eq:sg}
\partial_{tt}^2 u(x, t) - c^2 \partial_{xx}^2 u(x, t) +   \sin(u(x,t)) = 0 \, ,
\end{equation}
with $x \in \R$ and $t\in \R^+ \setminus \{0\}$ being the spatial and temporal coordinates, respectively, $u(x, t)$ the displacement field and $c$ the propagation speed. The choice of the nonlinear forcing term $\sin(u(x, t))$ allows for the formation of solitons. Long-range interactions and memory effects can be introduced through an additive integral term to $\sin(u(x, t))$ accounting for the global influence on the solution at a point~\cite{mivskinis2005nonlinear,alfimov1998solitary,wang2015infinitely}:
\begin{equation}
\label{eq:nsg}
\partial_{tt}^2 u(x, t) - c^2 \partial_{xx}^2 u(x, t) + \sin(u(x, t)) - \int_{-\infty}^{\infty} S(x - y) \sin(u(y, t)) \,dy = 0 \, .
\end{equation}
\(S(x - y)\) is a nonlocal kernel regulating the influence of the entire spatial domain on the wave dynamics. The specific form of such kernel defines range and character of nonlocal interactions. Eq.\eqref{eq:nsg} often exhibits interesting and rich dynamics, and its solutions can involve the formation of various types of localized structures, including the so-called nonlocal solitons~\cite{alfimov2005numerical,tang2015nonlocal,xiang2022local}. 

Indeed, nonlocality can be prescribed also for the wave internal response function $\partial_{xx} u(x, t)$ by promoting the local Laplace operator to an integral operator. This theory is called peridynamics and corresponds to a nonlocal mechanics continuum theory extending local elasticity to long-range interactions thus accounting for material discontinuities, damages and failures~\cite{Sill,Sill1,MR3831320, MR3297136, MR0307573, KRONER1967731, SILLING201073, MR2592410, MR2348150, MR2430855}. The peridynamic Sine-Gordon equation reads:

\begin{equation}
\label{psg}
\partial_{tt}^2 u(x, t) - \mathcal{L}(u(\cdot, t))(x) + \sin(u(x, t)) = 0\, ,
\end{equation}
where \(\mathcal{L}(u(\cdot, t))\) is the peridynamic kernel. As well-known, traditional continuum mechanics constitutive equations rely on spatial gradients of displacements and stresses. On the contrary, within peridynamic framework, the material internal energy is defined by integral (and often fractional) operators involving information from a region around a given point usually referred as \textit{peridynamic horizon}. Moreover, the integro-differential nature of Eq.~\eqref{psg} depicts the material dynamics at several scales at the same time. This feature lies in its ability to capture the effects of small-scale features on the overall behavior of the solution~\cite{Coclite_2021,CDFMV,coclite20232}. The synergy between nonlocal differential equations and multiscale modeling enhances the ability to simulate and understand complex physical phenomena occurring in materials and systems with intricate structures and behaviors across multiple scales. {In this context, any model can be generalized to a nonlocal model by re-thinking the differential operator involved in the local formulation. Indeed, the physically soundness of the formulation as well as its convergence to the local original formulation is demanded to the constitutive choices made for $\mathcal{L}(u(\cdot, t))(x)$.} The complexity of the peridynamic operator would often require the use of raffinate numerical tools for approximating the evolution of the system in time. In the last two decades the increasing interest in peridynamics grossly divided the researchers in two main directions: finite element models~\cite{lipton14,lipton16,jha18,jha19,jha21,lipton21} and mesh-free methods~\cite{coclite20242,silling2005meshfree,gu2017voronoi,bessa2014meshfree} or quadrature methods \cite{shojaei2022hybrid}. Only in the very last years spectral methods~\cite{jafarzadeh2020efficient,LP2021,LP2022,LPcheby2022,LPcheby,LPeigenv} and boundary element methods \cite{liang2021boundary} have been developed enlarging the range of suitable numerical tools. Indeed, due to the convergence properties of different numerical scheme, mesh-free and quadrature methods are commonly chosen when dealing with nonlinear kernels; on the contrary, spectral methods and boundary element methods are usually chosen for accurate numerical solutions of complex semilinear or linear peridynamics. As for finite elements discretization, the number of proposed schemes in literature is so large and so diverse that a main application is hard to find.
Moreover, the specific formulation of a peridynamic equation, including the choice of the kernel function and the underlying assumptions, may vary based on the requirements of the specific application or material being studied. As the matter of facts, by combining some grandstanding ingredients such as, nonlocality, fractionality, nonlinearity and singularity, the ability of a specific peridynamic kernel is defined~\cite{DCFP}.

In this paper we consider the fractional one-dimensional kernel proposed and discussed in~\cite{Coclite_2018}:
\begin{equation}
\label{eq:1}
\mathcal{L}u(x,t)=\int_{B_{\delta}(x)} f(x-x',u(x,t)-u(x',t))\,dx'=\int_{B_{\delta}(x)} \frac{u(x,t)-u(x',t)}{\left|x-x'\right|^{1+2\alpha}}\, dx'\, ,
\end{equation}
with $\alpha\in(0,1)$ and $B_\delta (x)$ being the interval centered on $x$ by radius $\delta$. 
This choice is motivated under the following four physically-sound constitutive assumptions for the pairwise force interaction $f:\R \times (\R \setminus \{0\}) \rightarrow \R$:
\begin{itemize}
\item $f\in C^1(\R\times (\R \setminus \{0\});\R)$\, ,
\item $f(x-x',u(x)-u(x'))=f(x'-x,u(x)-u(x'))$ for every $(x-x',u(\cdot,t))\in \R \times (\R \setminus \{0\})$\, ,
\item $f(x-x',u(x)-u(x'))=-f(x'-x,u(x')-u(x)$ for every $(x-x',u(\cdot,t))\in \R \times (\R \setminus \{0\})$\, ,
\item there exists a function $\Phi\in C^2(\R \times (\R \setminus \{0\}))$ such that $\Phi = \nabla_u f$\, .
\end{itemize}
Due to the choice made for the pairwise force interaction, the energy space is $H^{\alpha}$ and $\mathcal{L}u(x,t)$ corresponds to the censured--$\alpha$-Laplacian ($\alpha$ being the fractionality parameter). {This formulation has been proven to converge to the local Laplace operator~\cite{Coclite_2018,orlando2024}; in this light it represents the natural generalization of the local sine-Gordon equation.}

The analysis of the evolutionary problem associated to Eq.\eqref{psg} with local Neumann boundary condition is addressed here by adopting the spectral method approach on Chebyshev polynomials nodes while the temporal integration is demanded to the St{\"o}rmer-Verlet scheme, thus returning a suitable, efficient and accurate computational tool~\cite{LP2021,LP2022,LPeigenv}. This method is peculiarly suited for integral operator that can be expressed as convolution product so that preserving the computational properties of the Fast Fourier Transform (FFT) algorithm. Moreover, discretizing the spatial domain with trigonometric polynomials boundary conditions may be readily imposed~\cite{LPcheby2022,LPcheby}.

The paper is organized as follows. In Section~\ref{linmod} we present the evolutionary problem associated to the peridynamic Sine-Gordon equation and we prove the conservation of the energy of the model. In Section~\ref{sec:spectraldiscr} we propose a spectral discretization for the evolutionary problem and prove the convergence of the solutions of the semidicretized problem to the ones of the continuous problem. In Section~\ref{sec:fully} the full discretization of the semi-discrete problem is obtained by using the St{\"o}rmer-Verlet scheme. Section~\ref{sec:numerics} is devoted to the numerical simulations: in Section~\ref{sec:validation} we firstly critically analyze strength and limitation of the proposed spectral method by validating the latter against a consolidated {second}-order finite difference scheme within several benchmark tests. Such tests include soliton-like solutions. Then, in Section~\ref{sec:dispersive} we demonstrate the presence of dispersive effects on the solution of the nonlocal evolutionary problem and, lastly, {in Section~\ref{sec:energy} we experimentally study the energy behavior of the fully-discrete scheme. The numerical simulation shows a sort of \textit{almost-preserving} property of the fully-discrete scheme proposed in the sense that unless a certain initial time in which the energy functional presents strong oscillations, then the behavior of the discretized energy oscillates around 1 with a one percent error.} Conclusions and future outlines end the paper.

\section{The model}
\label{linmod}

We consider the nonlocal one-dimensional Sine-Gordon equation given by
\begin{equation}
\label{eq:nlsineG}
\ptt u(x,t) = \mathcal{L} u(x,t) -\sin\left(u(x,t)\right),
\end{equation}
defined on a compact domain $\Omega\subset \R$, where $\mathcal{L}$ is the integro-differential operator defined as
\begin{equation}
\label{eq:L}
\mathcal{L} u(x,t)=\int_{B_{\delta}(x)} \frac{u(x,t)-u(x',t)}{\left|x-x'\right|^{1+2\alpha}}\,dx',\qquad \alpha\in(0,1),
\end{equation}
where $B_{\delta}(x)$ is the $x$-centered ball with radius $\delta>0$, with initial conditions
\begin{equation}
\label{eq:initcond}
u(x,0)=u_0(x),\quad \pt u(x,0)=v_0(x).
\end{equation}

The well-posedness of the Cauchy problem~\eqref{eq:nlsineG}-\eqref{eq:initcond} in the energy space and in the framework of hyper-elastic constitutive assumptions can be found in~\cite{Coclite_2018}. Therein it is also showed that the energy associated to~\eqref{eq:nlsineG} is given by
\begin{equation}
\label{eq:energy}
\begin{split}
E[u](t)&= \frac{1}{2}\int_{\Omega} \left|\partial_t u(x,t)\right|^2 dx
-\frac{1}{4}\int_{\Omega}\int_{B_{\delta}(x)}\frac{\left(u(x,t)-u(x',t)\right)^2}{|x-x'|^{1+2\alpha}}dx'dx\\&-\int_{\Omega}\left(1-\cos\left(u(x,t)\right)\right) dx 
\end{split}
\end{equation}

\begin{lemma}[Energy preserving property]
\label{lm:energy}
The energy's functional defined in~\eqref{eq:energy} is preserved by time.
\end{lemma}

\begin{proof}
We have
\begin{equation}
\label{eq:derenergy}
\begin{split}
\frac{d}{dt}E[u](t)&=\int_\Omega \partial_t u(x,t)\partial_{tt}^2 u(x,t) dx\\
&\quad -\frac{1}{2}\int_\Omega\int_{B_{\delta}(x)} \frac{\left(u(x,t)-u(x',t)\right)}{\left|x-x'\right|^{1+2\alpha}}\left(\partial_t u(x,t)-\partial_t u(x',t)\right) dx'dx\\
&\quad +\int_\Omega\partial_t u(x,t)\sin\left(u(x,t)\right) dx\\
&=\underbrace{\int_\Omega\partial_t u(x,t)\left(\partial_{tt}^2 u(x,t) + \sin\left(u(x,t)\right)\right) dx}_{{}=:I_1} \\
&\quad -\frac{1}{2} \underbrace{\int_\Omega\int_{B_{\delta}(x)} \frac{\left(u(x,t)-u(x',t)\right)}{\left|x-x'\right|^{1+2\alpha}}\partial_t u(x,t) dx'dx}_{{}=:I_2}\\
&\quad +\frac{1}{2} \underbrace{\int_\Omega\int_{B_{\delta}(x)} \frac{\left(u(x,t)-u(x',t)\right)}{\left|x-x'\right|^{1+2\alpha}}\partial_t u(x',t) dx'dx}_{{}=:I_3}.
\end{split}
\end{equation}

We need to show that $I_3=-I_2$. If we use the constitutive assumptions on the pair-wise force intercation function, by making a change of variables and rearranging terms in $I_3$ we find
\begin{equation}
\label{eq:I3energy}
\begin{split}
I_3&=
\frac{1}{2} \int_\Omega\int_{B_{\delta}(0)} \frac{u(x,t)-u(x-x',t)}{\left|x'\right|^{1+2\alpha}}\partial_t u(x-x',t) dx'dx\\
&= \frac{1}{2} \int_\Omega\int_{B_{\delta}(0)} \frac{u(x,t)-u(x+x',t)}{\left|x'\right|^{1+2\alpha}}\partial_t u(x+x',t) dx'dx\\
&= - \frac{1}{2} \int_\Omega\int_{B_{\delta}(0)} \frac{u(x+x',t)-u(x,t)}{\left|x'\right|^{1+2\alpha}}\partial_t u(x+x',t) dx'dx\\
&= - \frac{1}{2} \int_\Omega\int_{B_{\delta}(x)} \frac{u(x,t)-u(x',t)}{\left|x-x'\right|^{1+2\alpha}}\partial_t u(x,t) dx'dx\\
&=-I_2.
\end{split}
\end{equation}
Sustituting~\eqref{eq:I3energy} into~\eqref{eq:derenergy} and using~\eqref{eq:nlsineG} and the definition of the integral operator $\mathcal{L}$ in~\eqref{eq:L}, we get the claim.
\end{proof}

\section{Spectral semi-discretization}
\label{sec:spectraldiscr}

{Pseudo-spectral methods are often used to study nonlinear wave phenomena~\cite{rev2018,rev2021}. They are based on the implementation of the fast-Fourier transform on equidistant collocation points and require the imposition of periodic boundary conditions. A different approach allowing to overcome the limitation of such boundary condition consists in considering the Chebyshev polynomials and the derivative matrix. Following this strategy, }the nonlocal Sine-Gordon equation~\eqref{eq:nlsineG} can be discretized in space by using Chebyshev polynomials.
The use of spectral methods allows us to obtain high accuracy in the profile of the solution and, moreover, this approach is typically used when the integral operator can be expressed in terms of convolution products~\cite{LPcheby,LPcheby2022,LPeigenv}. Indeed, Chebyshev method can exploit the properties of the FFT algorithm to compute efficiently such products. Additionally, the choice of Chebyshev polynomials within the family of trigonometric polynomials allows us to {impose more general boundary conditions.}

The method consists in the approximation of the solution $u(x,t)$ to~\eqref{eq:nlsineG} by a finite linear combination of Chebyshev polynomials of the first kind.

For simplicity we assume the spatial domain to be $[-1,1]$, but a more general spatial interval can be considered by an affine transformation. Moreover we assume to have no-flux boundary conditions
\begin{equation}
\label{eq:bc}
\px u(\pm 1,t)=0.
\end{equation}

We derive the semi-discrete model of~\eqref{eq:nlsineG} as follows.

Let $N>0$ be the total number of discretization points in the spatial domain $\Omega=[-1,1]$ and $x_h=\cos(\pi h/N)$, $h=0,\dots,N$ be the Chebyshev Gauss-Lobatto (CGL) points. We also define $\Omega_x=B_{\delta}(x)\cap[-1,1]$, for any $x\in[-1,1]$, {so that $\Omega_0=B_\delta(0)\cap[-1,1]$.}

Then, if we set
\begin{equation}
\label{eq:kernel}
k(x)=\frac{1}{\left|x\right|^{1+2\alpha}} \chi_{\Omega_0}(x),
\end{equation}
{where $\chi_{\Omega}$ denotes the characteristic function which takes value $1$ if $x\in\Omega$ and value $0$ otherwise, then} we can rewrite equation~\eqref{eq:nlsineG} as follows
\begin{equation}
\label{eq:model}
\ptt u(x,t) = \beta u(x,t) - \left(k\ast u\right)(x,t) - \sin\left(u(x,t)\right),
\end{equation}
where $\beta = \int_{B_\delta(0)} k(s)ds$.

We notice that $\beta$ is well-defined because due to the nonlocal nature of the model, it does not allow self-interactions among material points and as a  consequence the integrand $k$ is not discontinuous over the domain of integration. Hence, we can assume that the kernel $k$ is uniformly bounded.


We look for an approximation of the solution $u(x,t)$ in the form of a finite linear combination of Chebyshev polynomials $T_n(x)$
\begin{equation}
\label{eq:uN}
u^N(x,t)=\sum_{n=0}^N \tilde{u}_n(t) T_n(x),\quad x\in\Omega,\quad t>0,
\end{equation}
where the coefficients $\tilde{u}_n(t)$ represent the discrete Chebyshev coefficients given by\begin{equation}
\label{eq:fct}
\tilde{u}_n=\frac{1}{\gamma_n}\sum_{h=0}^N u(x_h)\ T_n(x_h) w_h,
\end{equation}
where $\gamma_n$ is a normalization constant defined by
\begin{equation}
\label{eq:gamma}
\gamma_n=\begin{cases}
\pi\quad& n=0,N\\
\frac{\pi}{2}\quad& n=1,\dots,N-1
\end{cases}    
\end{equation}
and
\begin{equation}
\label{eq:wh}
    w_h=\begin{cases}
    \frac{\pi}{2N}\quad& h=0,N\\
    \frac{\pi}{N}\quad& h=1,\dots,N-1.
    \end{cases}
\end{equation}
The presence of these constants is required to ensure the orthogonality property of the Chebyshev polynomials with respect to the weight function $w(x)=1/\sqrt{1-x^2}$.

We substitute $u$ by $u^N$ into~\eqref{eq:model}. Since Chebyshev transform, denoted here by $\mathcal{F}$, fulfills the same properties of Fourier transform, we can rewrite a convolution product in the physic space as a multiplication of the Chebyshev transform of each factor in the frequency space. Thus, on each interior collocation point $x_h$, $h=1,\dots,N-1$, equation~\eqref{eq:model} can be approximated as follows
\begin{equation}
\label{eq:semidiscrete}
\ptt u^N(x_h,t)=\beta u^N(x_h,t) - \mathcal{F}^{-1}\left(\mathcal{F}\left(k\right)\mathcal{F}\left(u^N\right)\right)(x_h,t) - \sin\left(u^N(x_h,t)\right).
\end{equation}
In the same way we approximate the initial conditions in~\eqref{eq:initcond} by
\begin{equation}
\label{eq:icdiscrete}
u^N(x_h,0)=u_0(x_h),\qquad \pt u^N(x_h,0)=v_0(x_h),\quad h=0,\dots,N.
\end{equation}

Following~\cite{Canuto86}, in order to impose Neumann boundary conditions to the discrete approximated solution
\begin{equation}
\partial_x u^N(x_0,t)=\partial_x u^N(x_N) =0,
\label{eq:bcdiscrete}
\end{equation} 
on each time level we have to solve the $2\times2$ system
\begin{equation}
\begin{cases}
d_{00}u^N(x_0,t) + d_{0N} u^N(x_N,t) &= -\sum_{h=1}^{N-1} d_{0h} u^N(x_h,t),\\
d_{NN} u^N(x_N,t) + d_{N0}u^N(x_0,t) &= -\sum_{h=1}^{N-1} d_{Nh} u^N(x_h,t),
\end{cases}
\label{eq:bcsystem}
\end{equation}
where $D=\left(d_{ij}\right)$, $i,\,j=0,\dots,N$ is the matrix representing the spectral derivative at the CGL collocation points (an explicit formula for the component of $D$ can be found in~\cite{Trefethen}). {For readers' convenience, we remark that equations~\eqref{eq:bcsystem} basically represent the multiplication of the first row and the last row of the derivative matrix $D$ by $u^N$. Indeed, the left-hand side refers to the derivative of $u^N$ at the boundary meshpoints $x_0$ and $x_N$.}

We can prove the convergence of the semi-discrete scheme~\eqref{eq:semidiscrete}-\eqref{eq:icdiscrete}-\eqref{eq:bcdiscrete} in the framework of weighted Sobolev space $H^s_w\left(\Omega\right)$, $s\ge1$.

We start by introducing the functional setting under consideration. In what follows, $C$ denotes a generic positive constant. We denote by $(\cdot,\cdot)_w$ and $\norm{\cdot}_w$ the inner product and the norm of $L^2_w\left(\Omega\right)$, respectively, namely
\[
(u,v)_w = \int_{\Omega} u(x)v(x)w(x)\,dx,\qquad \norm{u}_w^2 = (u,u)_w,
\]
with $w(x)=\left(\sqrt{1-x^2}\right)^{-1}$.

Let $s>0$, $H^s_w\left(\Omega\right)$ be the weighted Sobolev space and $X_s = \mathcal{C}^1\left(0,T; H^s_w\left(\Omega\right)\right)$ be the space of all continuous functions in $H_w^s\left(\Omega\right)$ whose distributional derivative is in $H_w^s\left(\Omega\right)$, with norm
\[
\norm{u}_{X_s}^2 = \max_{t\in[0,T]}\left(\norm{u(\cdot,t)}_w^2 + \norm{\pt u(\cdot,t)}_w^2\right),
\]
for any $T> 0$.

We introduce the space of Chebyshev polynomials of degree $N$ as follows
\[
S_N = \text{span}\left\{T_k(x)|-N \le k\le N\right\},
\]
and we define the projection operator $P_N: L^2_w(\Omega) \to S_N$ as
\[
P_Nu(x) = \sum_{|k|\le N} \tilde{u}_k T_k(x).
\]
It is such that for any $u\in L^2_w(\Omega)$, the following equality holds
\begin{equation}
\label{eq:orthogonal}
(u-P_Nu,\varphi)_w = 0,\quad\text{for every $\varphi\in S_N$}.
\end{equation}
We have that the projector operator $P_N$ commutes with derivatives in the distributional sense:
\[
\partial_x^q P_Nu = P_N\partial_x^q u,\quad\text{and}\quad\partial_t^q P_Nu = P_N\partial_t^q u.
\]

Then, spectral scheme~\eqref{eq:semidiscrete}-\eqref{eq:icdiscrete}-\eqref{eq:bcdiscrete} can be rewritten by using the projection $P_N$ in the following way
\begin{equation}
\label{eq:schemePN}
\begin{split}
&\ptt u^N = P_N \mathcal{L}(u^N) - P_N \sin\left(u^N\right),\\
&u^N(x,0) = P_N u_0(x),\quad \pt u^N(x, 0) = P_N v(x),\\
&\px u^N(\pm 1,t) =0,
\end{split}
\end{equation}
where $u^{N}(x,t)\in S_N$ for every $0\le t\le T$.

We recall the following lemma which is preliminary to our result.
\begin{lemma}[see~\cite{Canuto}]
\label{lm:sobolev}
Let $0\le \mu\le s$, if $u\in H_w^{s}\left(\Omega\right)$, then the following inequality holds
\begin{equation}
\label{eq:sobolev}
\norm{u-P_N u}_{H_w^{\mu}\left(\Omega\right)} \le C N^{\mu-s}\norm{u}_{H^s_w\left(\Omega\right)},
\end{equation}
for any positive constant $C$.
\end{lemma}

We can prove the following theorem.
\begin{theorem}
\label{th:convergence}
Let $s\ge 1$, and assume $u(x,t)\in X_s$ is the solution of the problem~\eqref{eq:nlsineG} with initial conditions $u_0$, $v\in H_w^s\left(\Omega\right)$, and $u^N(x,t)$ is the solution of the semi-discrete scheme~\eqref{eq:schemePN}. Then, there exists a positive constant $C = C(T)$, which does not depend on $N$, such that
\begin{equation}
\label{eq:order_conv}
\norm{u-u^N}_{X_1} \le C L(T) \left(\frac{1}{N}\right)^{s-1} \norm{u}_{X_s}.
\end{equation}
\end{theorem}

\begin{proof}

Let $s\ge 1$. Triangular inequality gives us
\begin{equation}
\label{eq:triangular}
\norm{u-u^N}_{X_1} \le \norm{u-P_N u}_{X_1} + \norm{P_N u - u^N}_{X_1}.
\end{equation}
Let we first focus on $\norm{u-P_N u}_{X_1}$. Thanks to Lemma~\ref{lm:sobolev} we have
\[
\norm{(u-P_N u)(\cdot,t)}_{H^1_w\left(\Omega\right)} \le C N^{1-s}\norm{u(\cdot,t)}_{H_w^s\left(\Omega\right)},
\]
and
\[
\norm{\pt (u-P_N u)(\cdot,t)}_{H_w^1\left(\Omega\right)} \le C N^{1-s}\norm{\pt u(\cdot,t)}_{H_w^s\left(\Omega\right)}.
\]
Therefore,
\begin{equation}
\label{eq:1term}
\norm{u-P_N u}_{X_1} \le C N^{1-s}\norm{u}_{X_s}.
\end{equation}

We now estimate the term $\norm{P_N u - u^N}_{X_1}$. We fix $\varphi= \pt \left(P_N u - u^N\right) \in S_N$ as test function. Subtracting~\eqref{eq:schemePN} from~\eqref{eq:nlsineG} and taking the inner product with the test function $\varphi$, we get
\begin{equation}
\label{eq:difference}
\begin{split}
0=&\underbrace{\int_{\Omega} \left(\ptt u(x,t)-\ptt u^N(x,t)\right)\pt \left(P_N u(x,t) - u^N(x,t)\right)w(x)\,dx}_{{}=:I_1}\\
 &- \underbrace{\int_{\Omega}\left(\mathcal{L}(u(x,t))-P_N \mathcal{L}(u^N(x,t))\right)\pt \left(P_N u(x,t)- u^N(x,t)\right)w(x)\,dx}_{{}=:I_2}\\ 
&+ \underbrace{\int_{\Omega}\left(\sin\left(u(x,t)\right)-P_N \sin\left(u^N(x,t)\right)\right)\pt \left(P_N u(x,t) - u^N(x,t)\right)w(x)\,dx}_{{}=:I_3}.
\end{split}
\end{equation}

We focus on $I_1$. The orthogonal condition~\eqref{eq:orthogonal} implies that
\[
\int_{\Omega}\left(\ptt u(x,t) - P_N \ptt u(x,t)\right)\pt \left(P_N u(x,t)-u^N(x,t)\right)w(x)\,dx = 0.
\]

Thus,
\begin{equation}
\label{eq:I1}
\begin{split}
I_1 &= \int_{\Omega}\left(\ptt u(x,t) - P_N \ptt u(x,t)\right)\pt \left(P_N u(x,t)-u^N(x,t)\right)w(x)\,dx\\
& + \int_{\Omega} \left(P_N \ptt u(x,t)-\ptt u^N(x,t)\right)\pt \left(P_N u(x,t)-u^N(x,t)\right)w(x)\,dx\\ 
&= \frac{1}{2}\frac{d}{dt}\norm{\pt (P_N u - u^N)(\cdot,t)}^2_{H_w^1\left(\Omega\right)}.
\end{split}
\end{equation}

Thanks to~\eqref{eq:orthogonal}, we have
\[
\int_{\Omega}\left(\mathcal{L}(u^N(x,t))-P_N \mathcal{L}(u^N(x,t))\right)\pt \left(P_N u(x,t)-u^N(x,t)\right)w(x)\,dx =0.
\]

Thus, since $k\in L^{\infty}\left(\Omega\right)$, using the Cauchy's inequality, we have
\begin{equation}
\label{eq:I2}
\begin{split}
I_2 &=\int_{\Omega} \left(\mathcal{L}(u(x,t))-\mathcal{L}(u^N(x,t))\right)\pt \left(P_N u(x,t)-u^N(x,t)\right)w(x)\,dx\\
&=\int_{\Omega}\int_{\Omega_x}k(x' - x)\left(u(x',t)-u(x,t)\right)\pt \left(P_N u(x,t)-u^N(x,t)\right)w(x)\,d x' dx\\
&\quad -\int_{\Omega}\int_{\Omega_x}k(x' - x)\left(u^N(x',t) - u^N(x,t)\right)\pt \left(P_N u(x,t)-u^N(x,t)\right)w(x)\,d x' dx\\
&\le \int_{\Omega}\int_{\Omega_x}k(x' - x)\left|u(x',t)-u^N(x',t) \right|\pt \left(P_N u(x,t)-u^N(x,t)\right)w(x)\,d x' dx\\
&\quad + \int_{\Omega}\int_{\Omega_x}k(x' - x)\left|u(x,t)- u^N(x,t)\right|\pt \left(P_N u(x,t)-u^N(x,t)\right)w(x)\,d x' dx\\
&\le \left(\frac{\beta}{2}+\norm{k}_{L^{\infty}\left(\Omega\right)}\right)\left(\norm{\left(u-u^N\right)(\cdot,t)}^2_{H_w^1\left(\Omega\right)} + \norm{\pt (P_N u - u^N)(\cdot,t)}^2_{H_w^1\left(\Omega\right)}\right).
\end{split}
\end{equation}

Moreover, the orthogonality condition~\eqref{eq:orthogonal} ensures that
\[
\int_{\Omega}\left(\sin(u^N(x,t))-P_N \sin(u^N(x,t))\right)\pt \left(P_N u(x,t)-u^N(x,t)\right)w(x)\,dx =0.
\]
Therefore, due to the uniformly Lipschitzianity of the loading term and the Cauchy's inequality, we find
\begin{equation}
\label{eq:I3}
\begin{split}
I_3 &=\int_{\Omega} \left(\sin\left(u(x,t)\right)-P_N\sin\left(u^N(x,t)\right)\right)\pt \left(P_N u(x,t)-u^N(x,t)\right)w(x)\,dx\\
&=\int_{\Omega}\left(\sin\left(u(x,t)\right)-\sin\left(u^N(x,t)\right)\right)\pt \left(P_N u(x,t)-u^N(x,t)\right)w(x)\, dx\\
&\le \frac{1}{2}\left(\norm{\left(u-u^N\right)(\cdot,t)}^2_{H_w^1\left(\Omega\right)} + \norm{\pt (P_N u - u^N)(\cdot,t)}^2_{H_w^1\left(\Omega\right)}\right).
\end{split}
\end{equation}

We substitute~\eqref{eq:I1},~\eqref{eq:I2} and~\eqref{eq:I3} in~\eqref{eq:difference} and we obtain

\begin{equation}
\label{eq:substitution}
\frac{1}{2}\frac{d}{dt}\norm{\pt (P_N u - u^N)(\cdot,t)}^2_{H_w^1\left(\Omega\right)} \le C \norm{\left(u-u^N\right)(\cdot,t)}^2_{H_w^1\left(\Omega\right)} + C \norm{\pt \left(P_N u - u^N\right)(\cdot,t)}^2_{H_w^1\left(\Omega\right)}.
\end{equation}
We add to both sides of equation~\eqref{eq:substitution} the term
\[
\frac{1}{2}\frac{d}{dt}\norm{\left(P_N u - u^N\right)(\cdot,t)}^2_{H_w^1\left(\Omega\right)} =\int_{\Omega}\left(P_N u (x,t) - u^N(x,t)\right)\pt \left(P_N u (x,t) - u^N(x,t)\right)w(x)\,dx,
\]
then, we find
\begin{equation*}
\begin{split}
&\frac{d}{dt}\left(\norm{\pt \left(P_N u - u^N\right)(\cdot,t)}^2_{H_w^1\left(\Omega\right)} + \norm{\left(P_N u - u^N\right)(\cdot,t)}^2_{H_w^1\left(\Omega\right)}\right)\\
&\le C\left(\norm{\pt \left(P_N u - u^N\right)(\cdot,t)}^2_{H_w^1\left(\Omega\right)} + \norm{\left(P_N u - u^N\right)(\cdot,t)}^2_{H_w^1\left(\Omega\right)} + \norm{\left(u-P_N u\right)(\cdot,t)}^2_{H_w^1\left(\Omega\right)}\right).
\end{split}
\end{equation*}
Since $\norm{\pt \left(P_N u - u^N\right)(\cdot,0)}_{H_w^1\left(\Omega\right)} = 0$ and $\norm{\left(P_N u - u^N\right)(\cdot,0)}_{H_w^1\left(\Omega\right)} = 0$, we can apply Lemma~\ref{lm:sobolev} and Gronwall's inequality obtaining
\begin{align*}
\biggl(\norm{\pt \left(P_N u - u^N\right)(\cdot,t)}^2_{H_w^1\left(\Omega\right)} &+ \norm{\left(P_N u - u^N\right)(\cdot,t)}^2_{H_w^1\left(\Omega\right)}\biggr)\\
&\le\int_0^t e^{C(t-\tau)}\norm{\left(u-P_N u\right)(\cdot,\tau)}^2_{H_w^1\left(\Omega\right)}\,d\tau\\
&\le C(T)N^{2-2s}\int_0^t\norm{u(\cdot,\tau)}^2_{H_w^1\left(\Omega\right)}\,d\tau.
\end{align*}
Hence,
\begin{equation}
\label{eq:2term}
\norm{P_N u - u^N}_{X_1}\le C(T)N^{1-s}\norm{u}_{X_s}.
\end{equation}
Finally, using~\eqref{eq:1term} and~\eqref{eq:2term} in~\eqref{eq:triangular}, we get the claim.

\end{proof}

\section{The fully discrete scheme}
\label{sec:fully}

The semi-discrete nonlocal formulation of Sine-Gordon equation in~\eqref{eq:semidiscrete} can be integrated in time by using explicit forward and backward difference techniques. {St{\"o}rmer-Verlet method consists in a symplectic semi-implicit centered second-order finite difference scheme deeply used in the context of continuum mechanics and elastodynamics.} 

Let {$N_T>0$} be the total number of time steps and 
$\{t_s\}_{s=0}^{N_T}$ be a uniform partition of the computational interval $[0,T]$, with $t_s=s\Delta t$ and $\Delta t=\frac{T}{N_T}$.

If we denote by
{
\begin{equation}
\label{eq:Ldisc}
\mathcal{L}\left(u^N(\cdot,t)\right)(x_h) = \beta u^N(x_h,t) - \mathcal{F}^{-1}\left(\mathcal{F}(k)\mathcal{F}\left(u^N\right)\right)(x_h,t)
\end{equation}}
the spectral discretization of the nonlocal operator $\mathcal{L}$ at the collocation points $x_h$, $h=0,\dots, N$, then the semi-discrete scheme~\eqref{eq:semidiscrete} can be compactly written as
\begin{equation}
\partial_{tt}^2 u^N(x_h,t) = \mathcal{L}\left(u^N(x_h,t)\right) -\sin\left(u^N(x_h,t)\right),\quad h=0,\dots,N.
\label{eq:semidiscompact}
\end{equation}

We set $U^{N,s}_h = u^N(x_h,t_s)$ the approximation of the solution at the node $x_h$ and at the discrete time $t_s$, for $h=0,\dots,N$ and $s=0,\dots,N_T$ and we denote $V^{N,s}_h=\left(U^{N,s}_h\right)'$, then the St\"ormer-Verlet scheme for the system~\eqref{eq:semidiscompact} is given by
\begin{equation}
\label{eq:fullyscheme}
\begin{cases}
U^{N,s+1}_h=U^{N,s}_h + \Delta t\left(V^{N,s}_h +\frac{\Delta t}{2} \mathcal{L}\left(U^{N,s}_h\right)-\frac{\Delta t}{2}\sin(U^{N,s}_h)\right),\\
V^{N,s+1}_h=V^{N,s}_h+\frac{\Delta t}{2}\left(\mathcal{L}\left(U^{N,s}_h\right)-\sin\left(U^{N,s}_h\right)+\mathcal{L}\left(U^{N,s+1}_h\right)-\sin\left(U^{N,s+1}_h\right)\right),
\end{cases}
\end{equation}
for $h=0,\dots,N$ and $s=0,\dots,N_T$.

\section{Numerical experiments}
\label{sec:numerics}

In this section we perform several simulations in order to study the properties of the solution of the model. The first test provides a validation of the spectral method by making a comparison between the solution obtained by Chebyshev method with the solution obtained by implementing a centered second order finite difference scheme. Then we show the dispersive effects due to the nonlocality in the profile of soliton-type solutions. {Finally, we point out the energy behavior of the fully-discrete scheme by showing that, despite the continuous energy functional is preserved over time, the proposed fully-discrete scheme can be considered \textit{almost-energy preserving}, in the sense that after certain initial oscillations, the energy functional recovers an averaged constant value.}

\subsection{Test 1: Validation of the performance of the spectral method}
\label{sec:validation}

In order to show that the spectral method approximates the correct solution, we make a comparison between the solution provided by the proposed Chebyshev method and the solution found by the implementation of a finite difference discretization. {The latter corresponding to a centered second-ordered scheme in time explicitly combined with the trapezoidal rule approximation of the operator $\mathcal{L} u(x,t)$.} We choose $u_{0,1}(x)=0$ and $v_{0,1}(x)=4 \left(\sqrt{1-c^2}\cosh(x/\sqrt{1-c^2})\right)^{-1}$ as initial conditions, with c=.999. {We fix $\alpha=0.4$ as parameter of the integral operator $\mathcal{L}$ defined in~\eqref{eq:1} and $\delta=0.2$ as horizon.} In Figure~\ref{fig:comp-kinkantikink}, we can observe a good agreement between the two numerical solutions.

\begin{figure}%
\centering
\includegraphics[width=.5\textwidth]{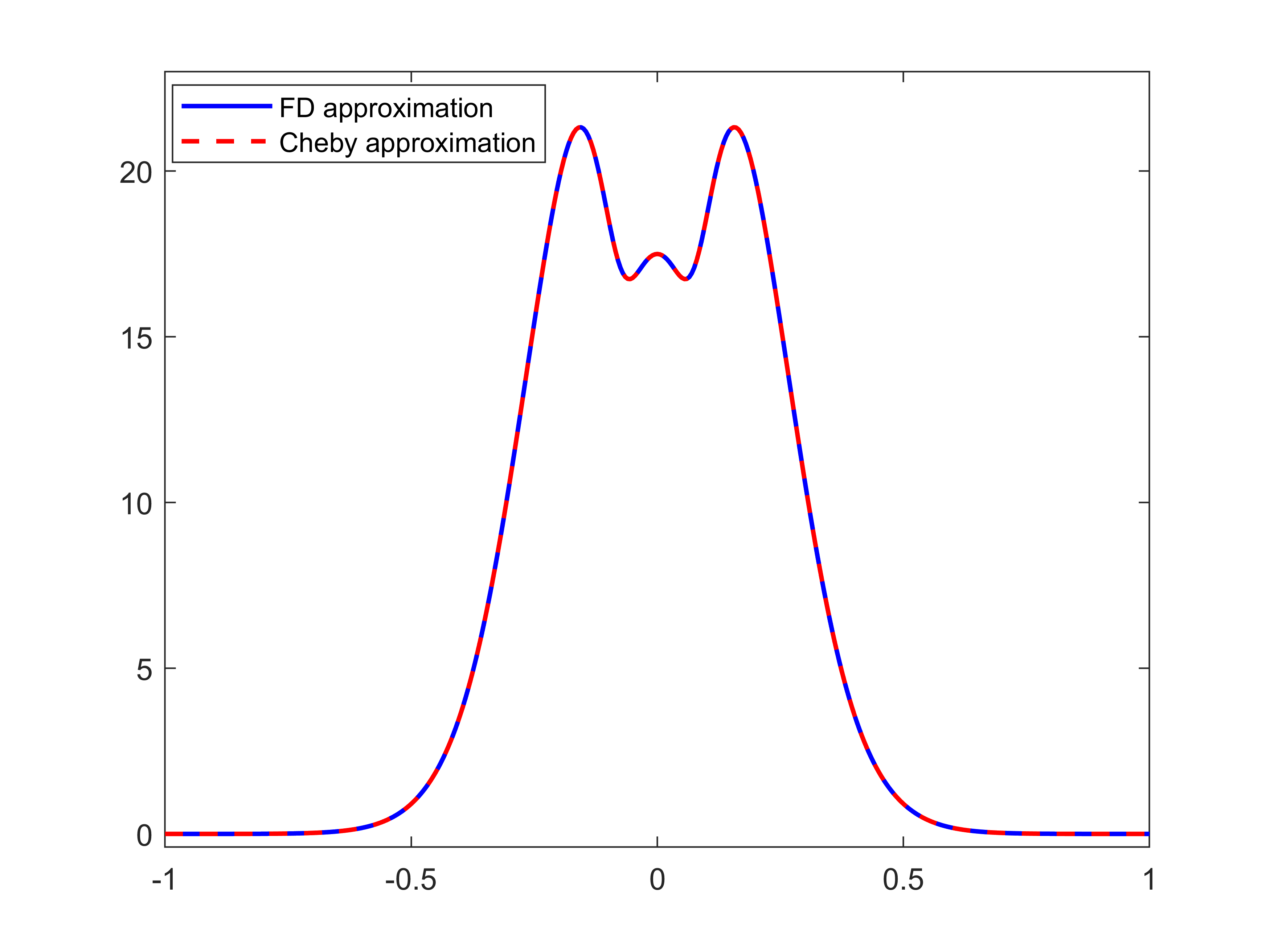}%
\caption{With reference to Section~\ref{sec:validation}, the comparison between the approximated solution computed by spectral method with the one obtained by finite different discretization at time $t=1$ with initial conditions $u_{0,1}(x)$ and $v_{0,1}(x)$. The parameters for the simulation are $N=400$, $N_T=800$, the horizon is $\delta=0.2$ and $\alpha=0.4$.}%
\label{fig:comp-kinkantikink}%
\end{figure}

In Figure~\ref{fig:compkink_t1} we show the comparison between the solution obtained by spectral method and finite difference method corresponding to the initial conditions $u_{0,2}(x)=4\arctan{\left(e^{\frac{x}{\sqrt{1-c^2}}}\right)}$ and $v_{0,2}(x)=-2c\frac{\mbox{sech}\left(\frac{x}{\sqrt{1-c^2}}\right)}{\sqrt{1-c^2}}$, with $c=0.999$.

\begin{figure}%
\centering
\includegraphics[width=.5\textwidth]{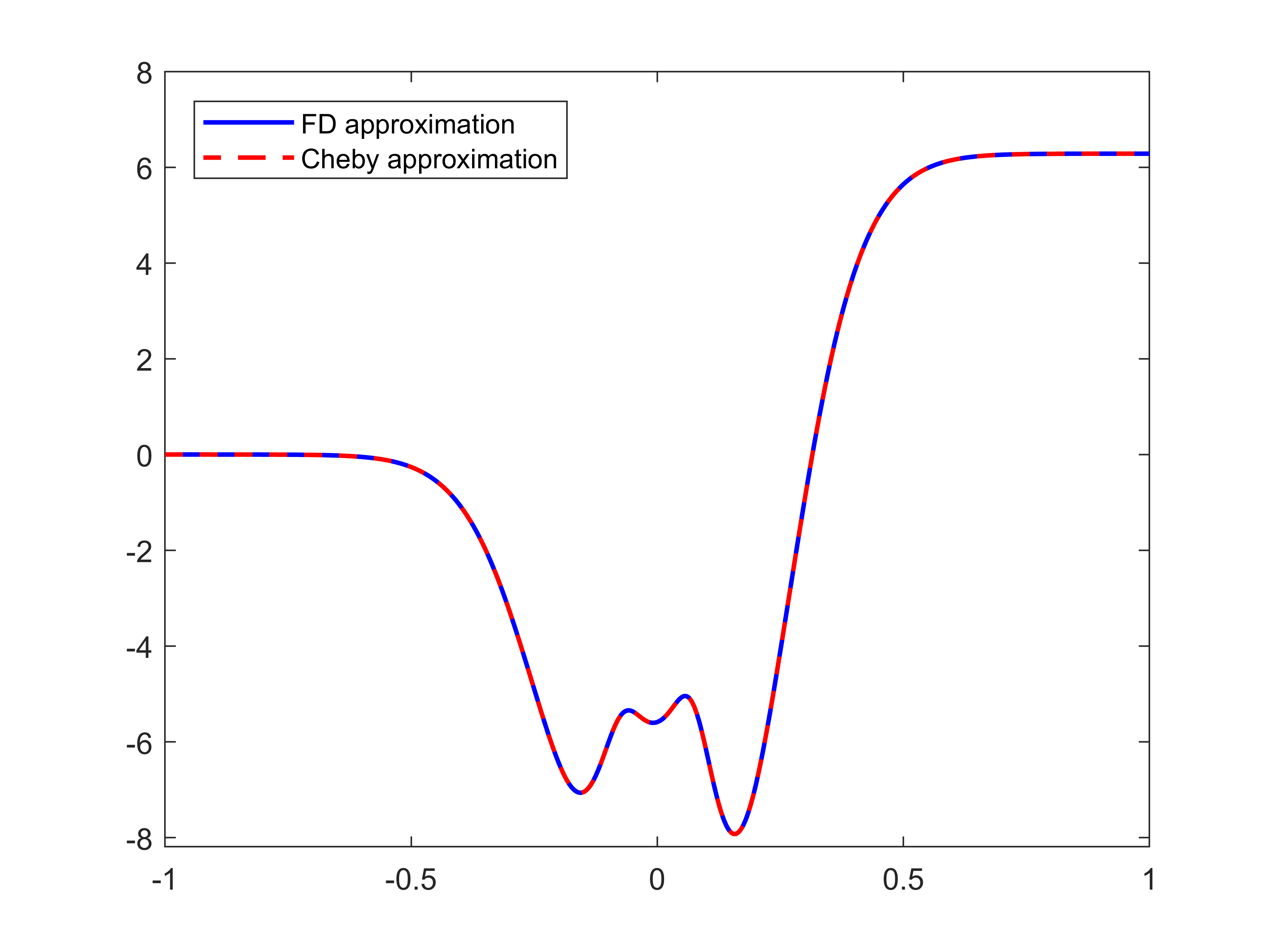}%
\caption{With reference to Section~\ref{sec:validation}, the comparison between the approximated solution computed by spectral method with the one obtained by finite different discretization at time $t=1$ with initial conditions $u_{0,2}(x)$ and $v_{0,2}(x)$. The parameters for the simulation are $N=400$, $\Delta t=8/N_T$, $N_T=800$ $\delta=0.2$ and $\alpha=0.4$.}%
\label{fig:compkink_t1}%
\end{figure}

With reference to the solution corresponding to the initial conditions $u_{0,2}(x)$ and $v_{0,2}(x)$, we also provide a convergence analysis of the proposed spectral scheme. We define the relative $L^2$-error as follows
\begin{equation*}
Error_2\left[u^N\right](t)=\frac{\sum_{h=1}^N\left|u^N(x_h,t)-u^\ast(x_h,t)\right|^2}{\sum_{h=1}^N\left|u^\ast(x_h,t)\right|^2},
\end{equation*}
where $u^\ast$ represents the reference solution obtained by using the finite difference scheme with $N=1600$. Starting from the relative $L^2$-error, we can compute the convergence rate by looking for the slope of the line that best fit in the sense of least square the logarithm of the data.

Table~\ref{tab:convrate} and Figure~\ref{fig:error_slope} depict the relative $L^2$-error and the convergence rate obtained by using the finite difference scheme and the proposed spectral method as the total number of spatial discretization points increases.

\begin{table}%
\centering%
\renewcommand\arraystretch{1.3}
\begin{tabular}{ccccc
}
\hline
\multirow{2}*{$N$}& \multicolumn{2}{c}{Finite Difference method}&\multicolumn{2}{c}{Chebyshev method}  \\ \cline{2-5}
& error & conv. rate & error&conv. rate\\ \hline

$100$&$2.1336 \times 10^{-3}$&$-$&$1.3600 \times 10^{-3}$&$-$
\\
$200$&$4.7141 \times 10^{-4}$&$2.1625$&$1.9584\times 10^{-4}$&$2
7757$
\\
$400$&$1.0768\times 10^{-4}$&$2.1426$&$7.1261\times 10^{-6}$&$3.7670$
\\
$800$&$1.8644\times 10^{-5}$&$2.2551$&$3.8717\times 10^{-7}$&$3.9945$
\\
\hline
\end{tabular}
\renewcommand\arraystretch{1}
\caption{With reference to Section~\ref{sec:validation}, the relative error related to the initial conditions $u_{0,2}(x)$ and $v_{0,2}(x)$, at time $t=2$ as function of the total number of discretization points.}
\label{tab:convrate}
\end{table}

\begin{figure}%
\centering
\includegraphics[width=.5\textwidth]{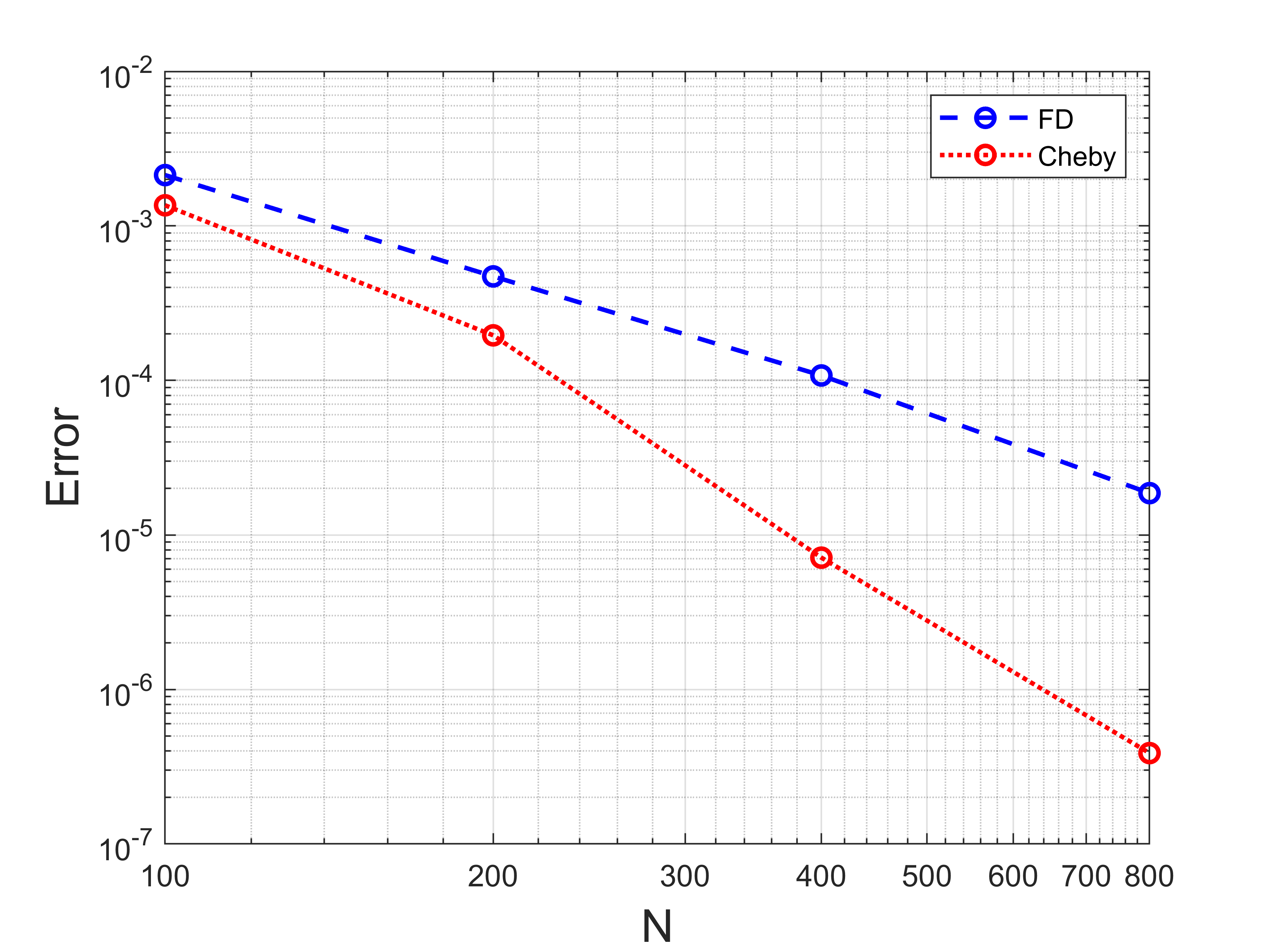}%
\caption{With reference to Section~\ref{sec:validation}, the comparison between the relative error obtained with the two considered methods in the logaritmic scale.}%
\label{fig:error_slope}%
\end{figure}

We can observe a gain in terms of convergence rate when a spectral method is applied.

\subsection{Test 2: The dispersive effects of the nonlocal model in soliton-type solutions}
\label{sec:dispersive}

Due to the presence of long-range interactions, solutions of~\eqref{eq:nlsineG} are characterized by a dispersive behavior. As a consequence, soliton-type solutions loose the property to be travelling waves with constant velocity and shape-preserving profile and show the appearance of an oscillatory behavior, whose phase depends on the size of the horizon.

The nonlocal kink solution is compared with the classical one in Figure~\ref{fig:kink}, where we can observe the evolution of the solution subject to the following initial conditions
\begin{align*}
u_0(x)&=4\arctan{\left(e^{\frac{x}{\sqrt{1-c^2}}}\right)},\\
v_0(x)&=-2c\frac{\mbox{sech}\left(\frac{x}{\sqrt{1-c^2}}\right)}{\sqrt{1-c^2}},
\end{align*}
with velocity $c=.999$.
\begin{figure}
\centering
\begin{subfigure}[b]{.48\textwidth}
\includegraphics[width=\textwidth]{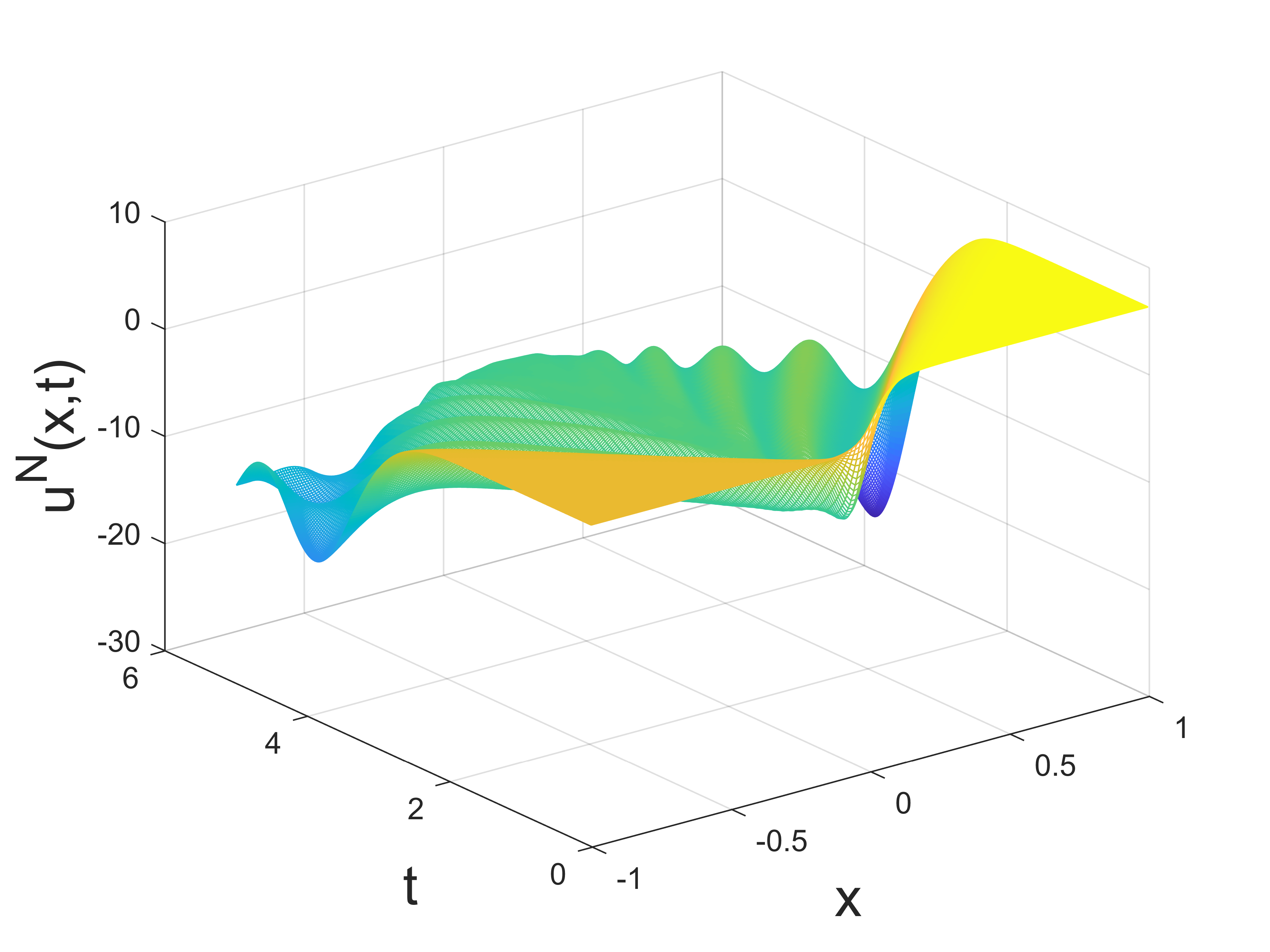}
\caption*{The evolution of a kink soliton in the nonlocal model.}
\end{subfigure}
\begin{subfigure}[b]{.48\textwidth}
\includegraphics[width=\textwidth]{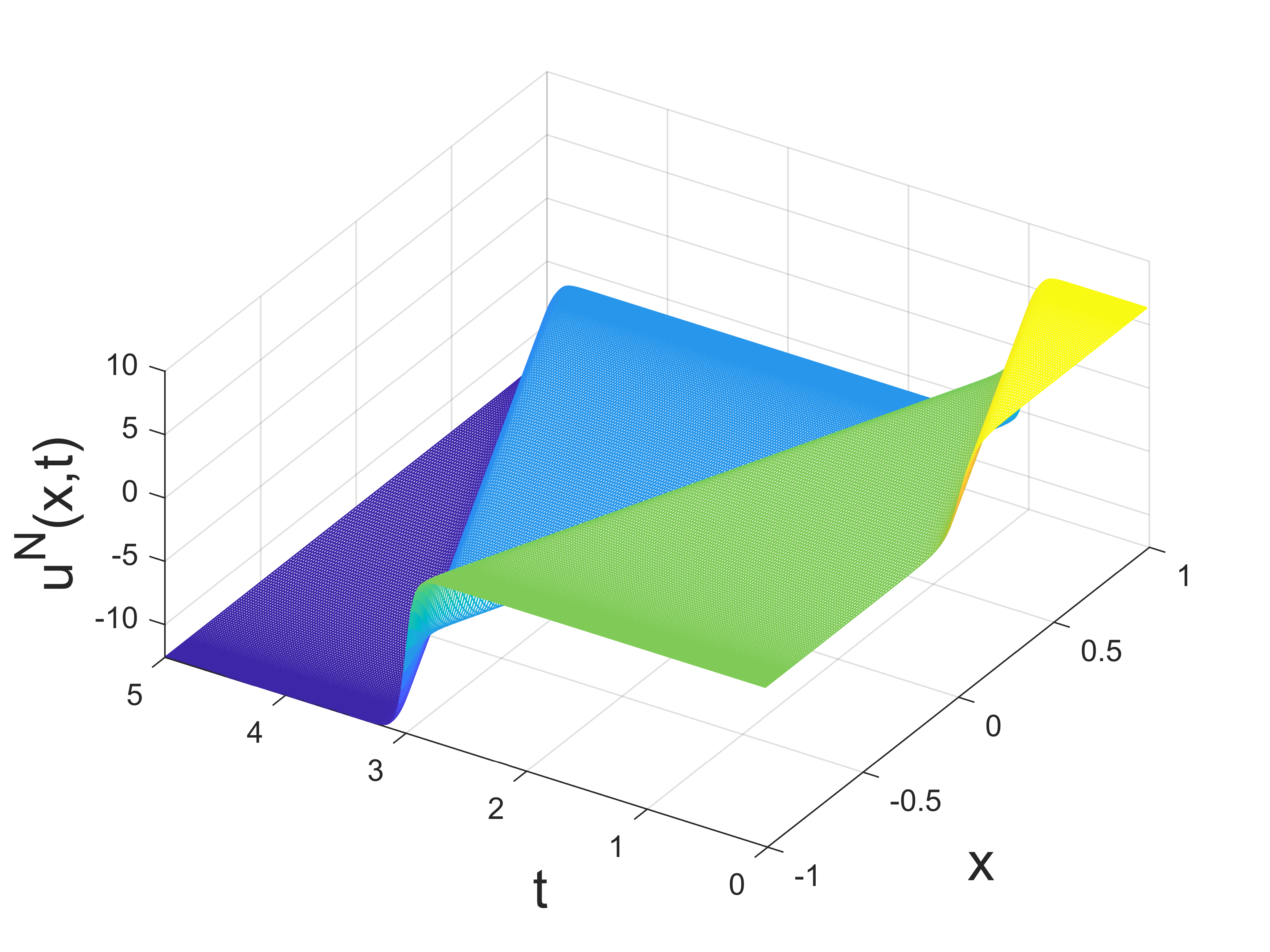}
\caption*{The evolution of a kink soliton in the classical SineGordon equation.}
\end{subfigure}
\caption{With reference to Section~\ref{sec:dispersive}: the comparison between kink-type solitons in the nonlocal model and in the classical Sine-Gordon model. The parameters for the simulation are $\delta=0.2$, $N=200$, $\alpha=0.4$, and $N_T=400$.}
\label{fig:kink}
\end{figure}

A comparison between antikink solutions is shown in Figure~\ref{fig:antikink}. The initial conditions for this test are
\begin{align*}
u_0(x)&=4\arctan{\left(e^{\frac{-x}{\sqrt{1-c^2}}}\right)},\\
v_0(x)&=-2c\frac{\mbox{sech}\left(\frac{x}{\sqrt{1-c^2}}\right)}{\sqrt{1-c^2}},
\end{align*}
with $c=0.999$.
\begin{figure}
\centering
\begin{subfigure}[b]{.48\textwidth}
\includegraphics[width=\textwidth]{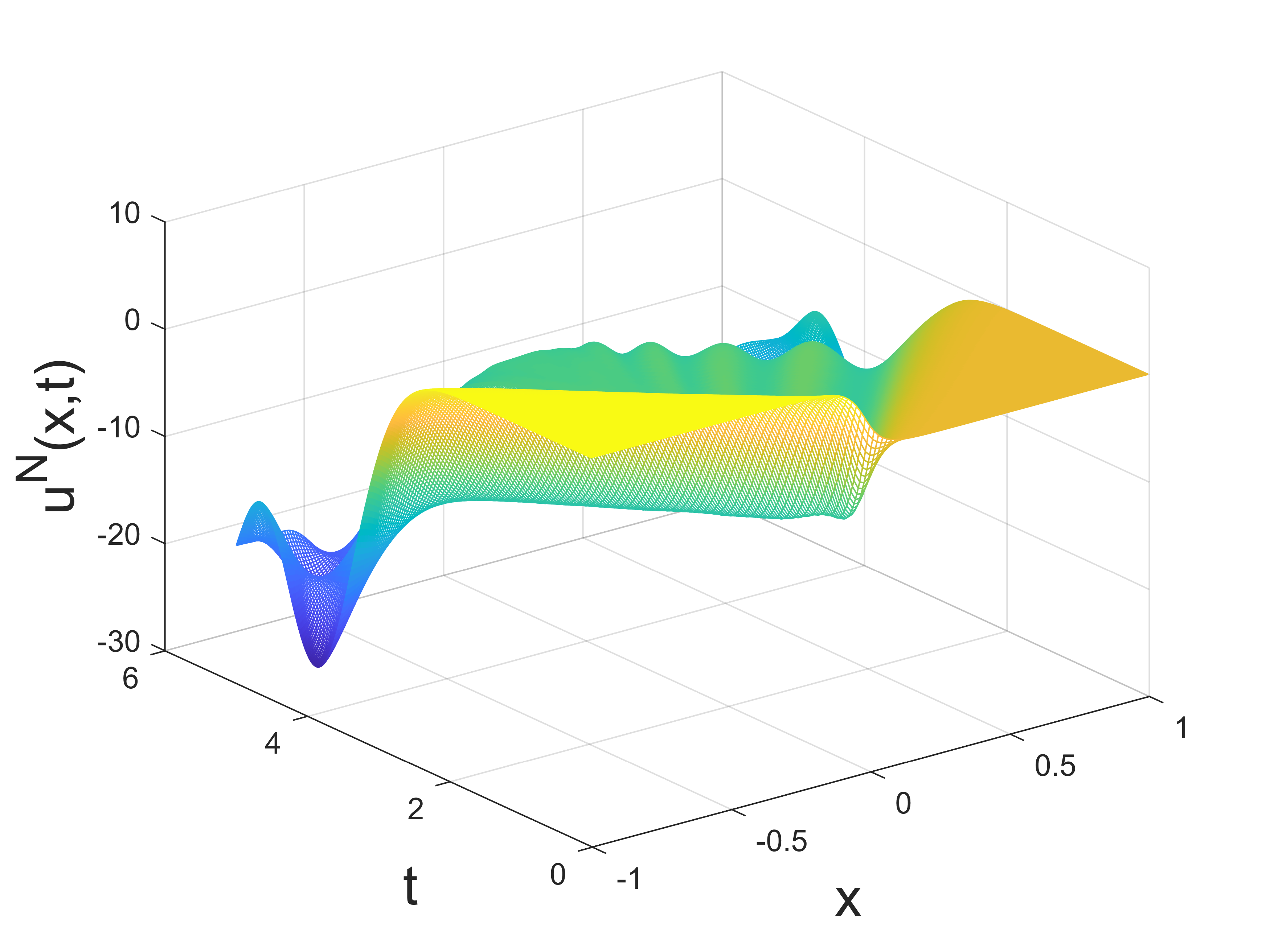}
\caption*{The evolution of a antikink soliton in the nonlocal model.}
\end{subfigure}
\begin{subfigure}[b]{.48\textwidth}
\includegraphics[width=\textwidth]{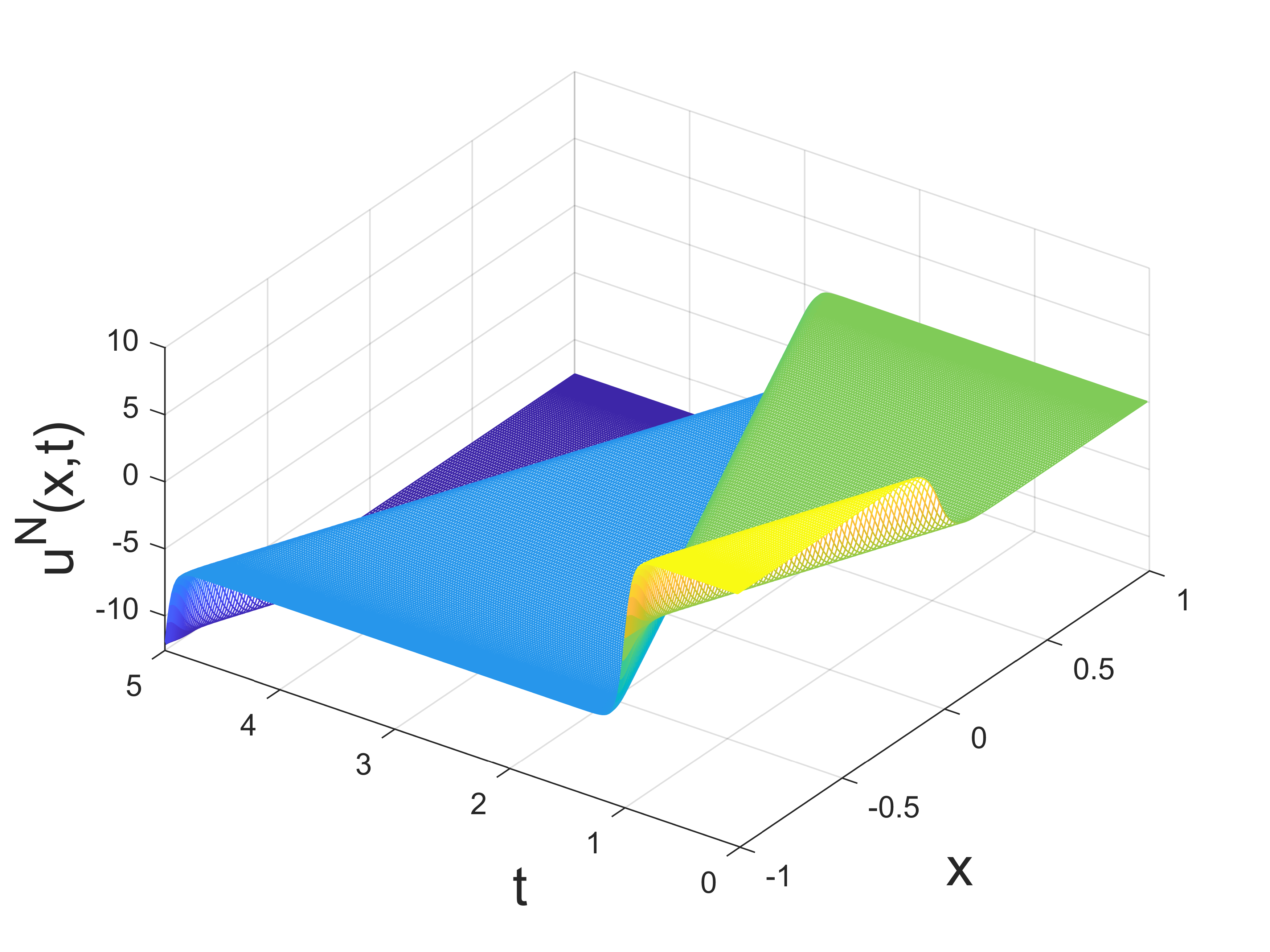}
\caption*{The evolution of a antikink soliton in the classical Sine-Gordon equation.}
\end{subfigure}
\caption{With reference to Section~\ref{sec:dispersive}: the comparison between antikink-type solitons in the nonlocal model and in the classical Sine-Gordon model. The parameters for the simulation are $\delta=0.2$, $N=200$, $\alpha=0.4$, and $\Delta t=8/N_T$, $N_T=400$.}
\label{fig:antikink}
\end{figure}

Figures~\ref{fig:kinkkink} and~\ref{fig:kinkantikink} provide a comparison between local and nonlocal models in the case of a collision between two kink-type solitons or a collision between a kink and an antikink soliton, respectively. The initial conditions for these simulations are respectively
\begin{align*}
u_0(x)&=4\arctan{\left(c \sinh{\frac{x}{\sqrt{1-c^2}}}\right)},\\
v_0(x)&=0,
\end{align*}
and
\begin{align*}
u_0(x)&=0,\\
v_0(x)&=\frac{4}{\sqrt{1-c^2}\cosh{\left(\frac{x}{\sqrt{1-c^2}}\right)}},
\end{align*}
with $c=0.999$.
\begin{figure}
\centering
\begin{subfigure}[b]{.48\textwidth}
\includegraphics[width=\textwidth]{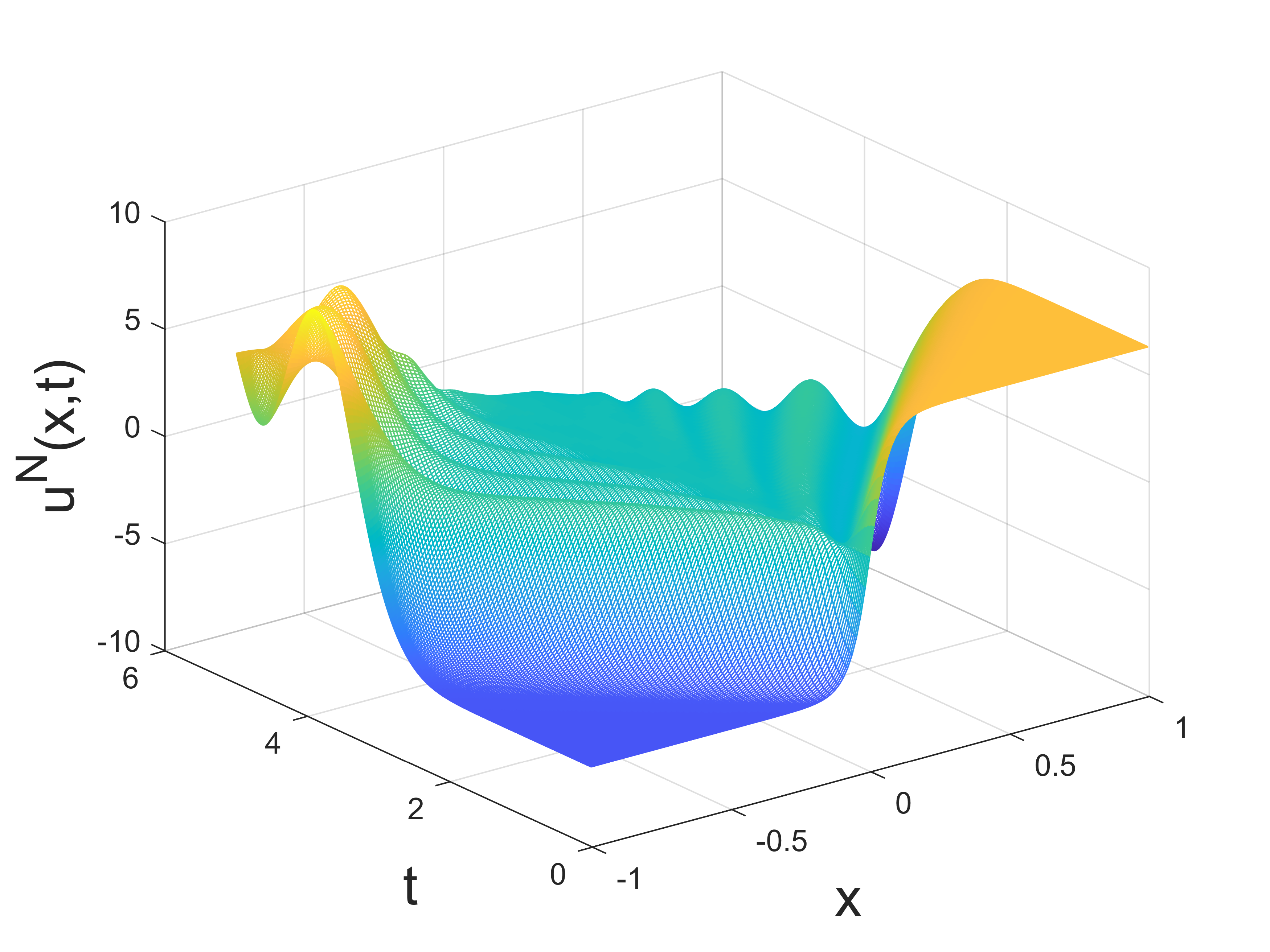}
\caption*{The evolution of a collision of two kink solitons in the nonlocal model.}
\end{subfigure}
\begin{subfigure}[b]{.48\textwidth}
\includegraphics[width=\textwidth]{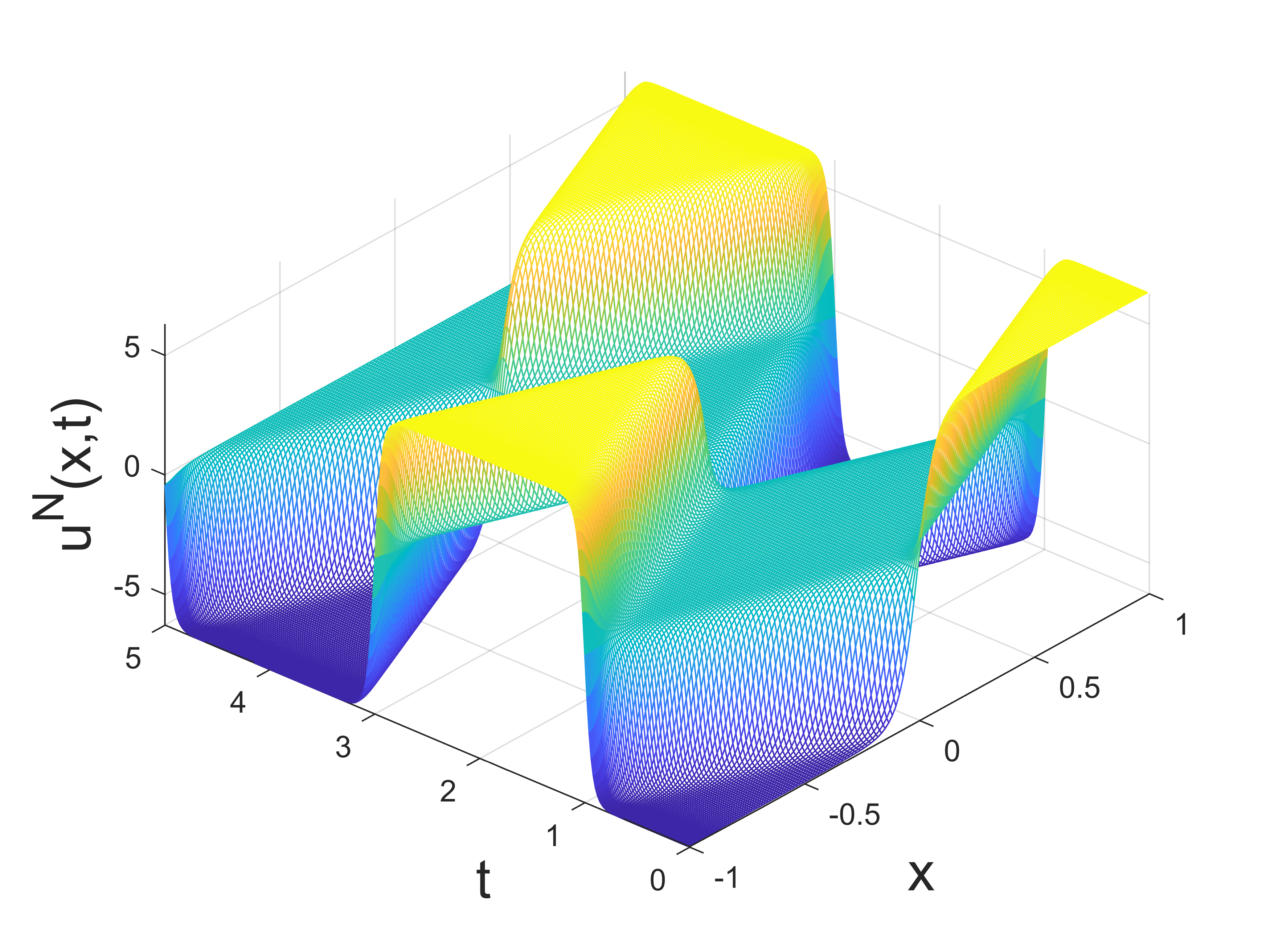}
\caption*{The evolution of a collision of two kink solitons in the classical Sine-Gordon equation.}
\end{subfigure}
\caption{With reference to Section~\ref{sec:dispersive}: the comparison between the behavior of a collision of two kink-type solitons in the nonlocal model and in the classical Sine-Gordon model. The parameters for the simulation are $\delta=0.2$, $N=200$, $\alpha=0.4$, and $\Delta t=8/N_T$, $N_T=400$.}
\label{fig:kinkkink}
\end{figure}

\begin{figure}
\centering
\begin{subfigure}[b]{.48\textwidth}
\includegraphics[width=\textwidth]{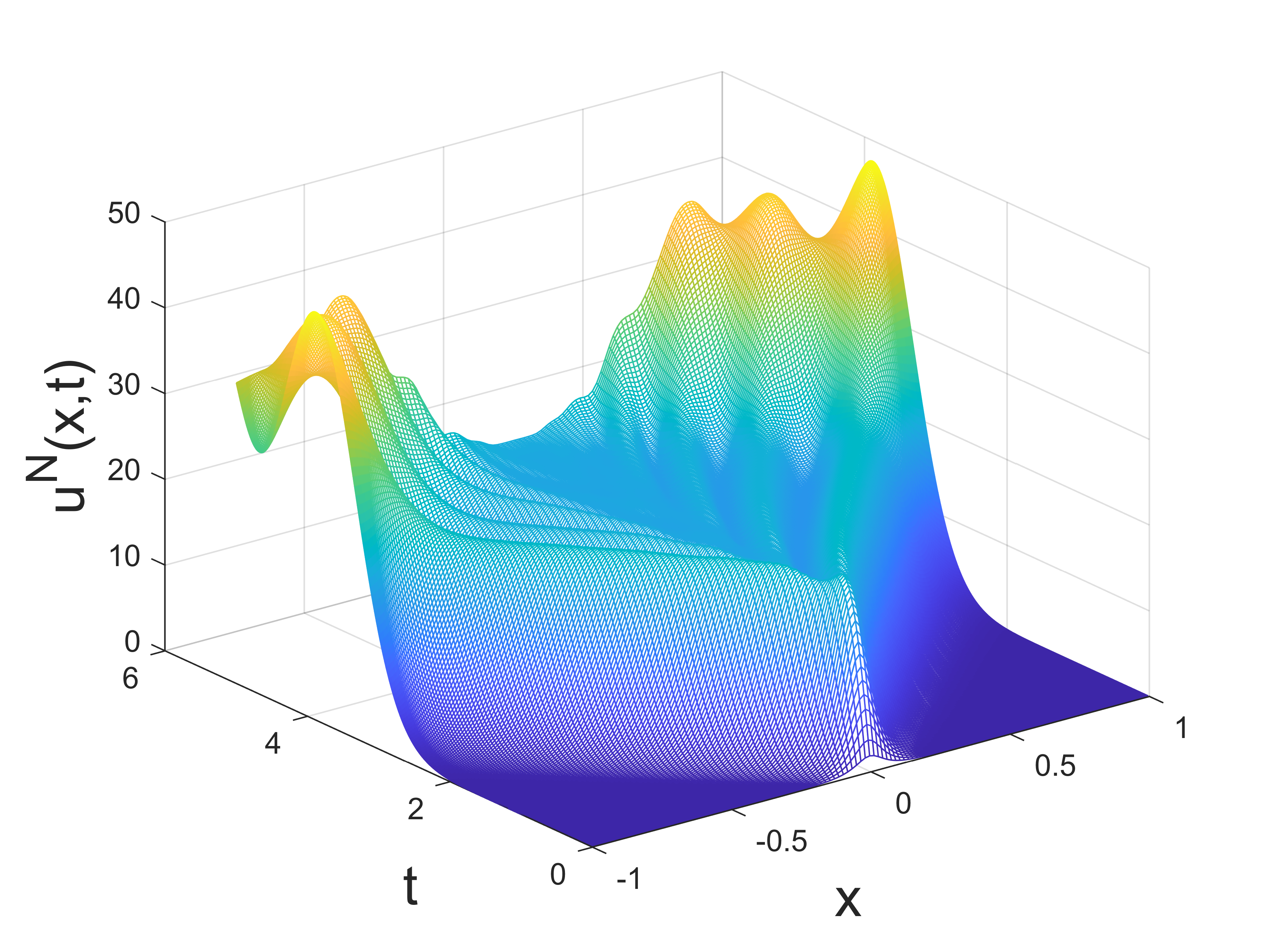}
\caption*{The evolution of a collision of a kink soliton with an antikink soliton in the nonlocal model.}
\end{subfigure}
\begin{subfigure}[b]{.48\textwidth}
\includegraphics[width=\textwidth]{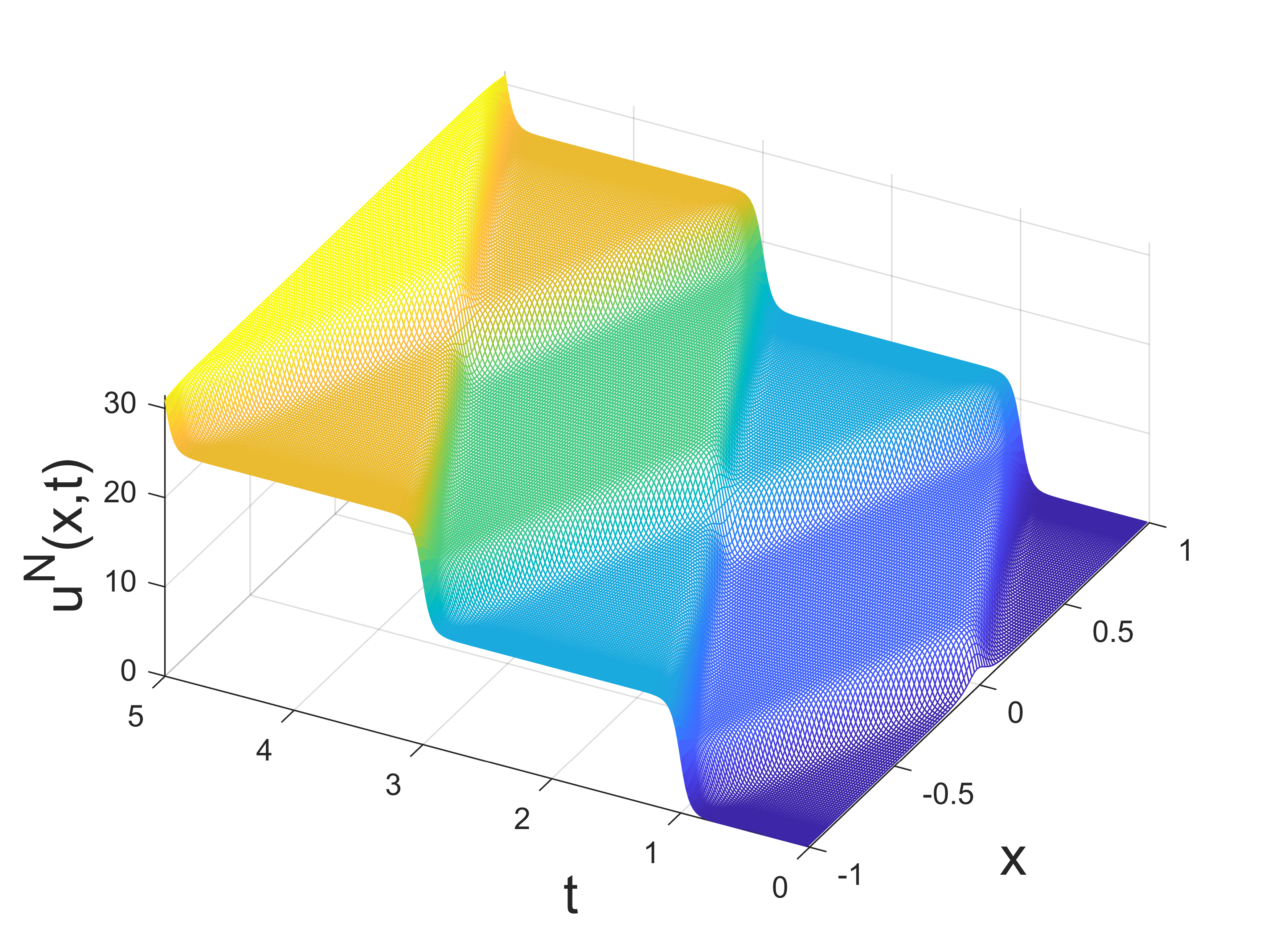}
\caption*{The evolution of a collision of a kink soliton with an antikink soliton in the classical Sine-Gordon equation.}
\end{subfigure}
\caption{With reference to Section~\ref{sec:dispersive}: the comparison between the behavior of a collision of kink-type soliton with an antikink-type soliton in the nonlocal model and in the classical Sine-Gordon model. The parameters for the simulation are $\delta=0.2$, $N=200$, $\alpha=0.4$, and $\Delta t=8/N_T$, $N_T=400$.}
\label{fig:kinkantikink}
\end{figure}

Finally, the same dispersive effects can be observed even in the behavior of a breather solution as shown in Figure~\ref{fig:breather}. In this case the initial conditions for the simulation are
\begin{gather*}
u_0(x)=4\arctan{\left(\frac{\sqrt{1-w^2}\sin\left(-\frac{c w x}{\sqrt{1-c^2}}\right)}{w\cosh\left(\frac{x\sqrt{1-w^2}}{\sqrt{1-c^2}}\right)}\right)},\\
v_0(x)=4w \frac{\sqrt{1-w^2}}{\sqrt{1-c^2}} \left(\frac{w\cos\left(-\frac{cwx}{\sqrt{1-c^2}}\cosh\left(\frac{x\sqrt{1-w^2}}{\sqrt{1-c^2}}\right)+c\sqrt{1-w^2}\sin\left(-\frac{cwx}{\sqrt{1-c^2}}\right)\sinh\left(\frac{x\sqrt{1-w^2}}{\sqrt{1-c^2}}\right)\right)}{w^2\cosh^2\left(\frac{x\sqrt{1-w^2}}{\sqrt{1-c^2}}\right)+(1-w^2)\sin^2\left(-\frac{cwx}{\sqrt{1-c^2}}\right)}\right),
\end{gather*}
with $c=0.999$ and $w=0.4$.
\begin{figure}
\centering
\begin{subfigure}[b]{.48\textwidth}
\includegraphics[width=\textwidth]{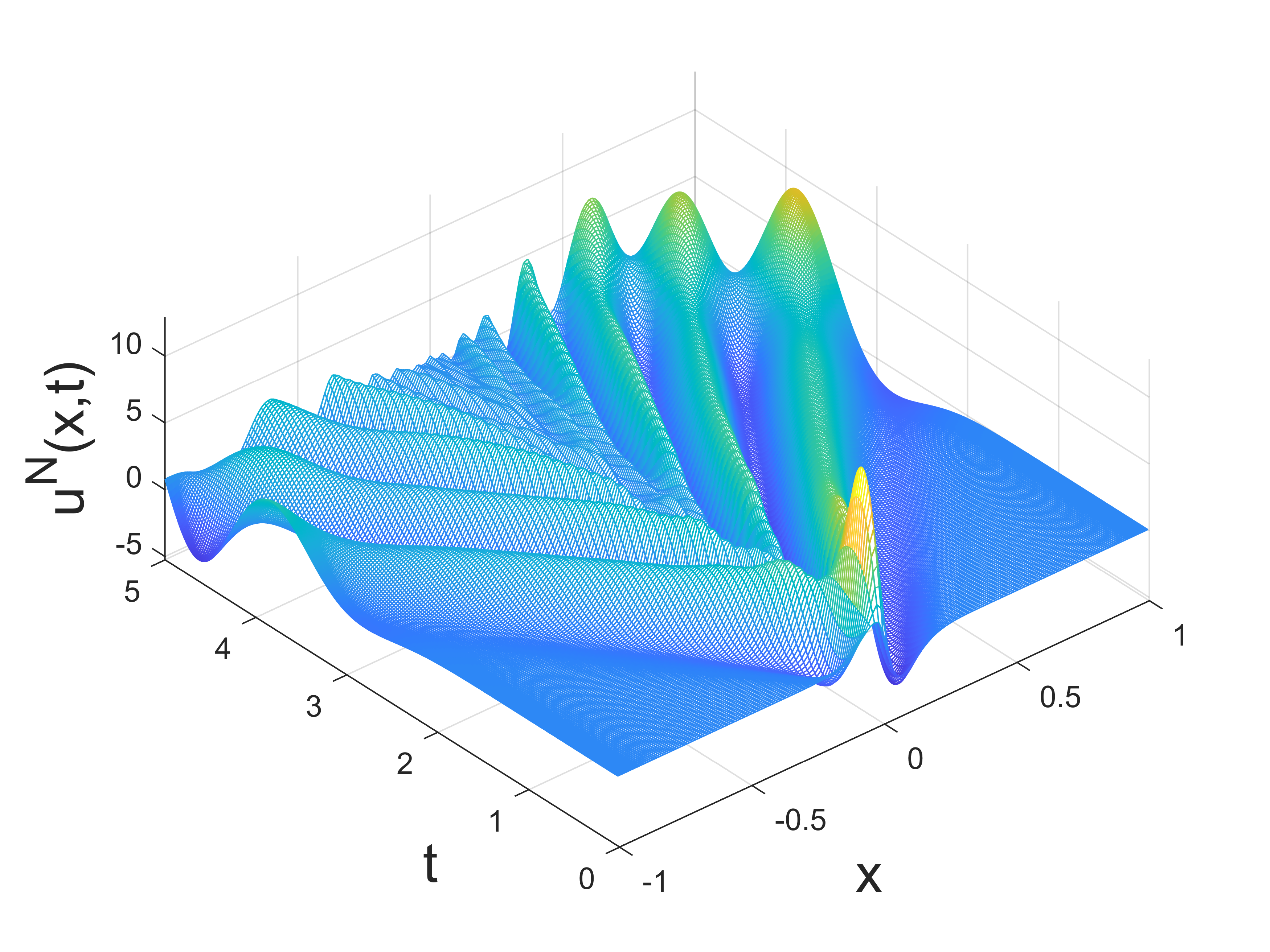}
\caption*{The evolution of a breather soliton in the nonlocal model.}
\end{subfigure}
\begin{subfigure}[b]{.48\textwidth}
\includegraphics[width=\textwidth]{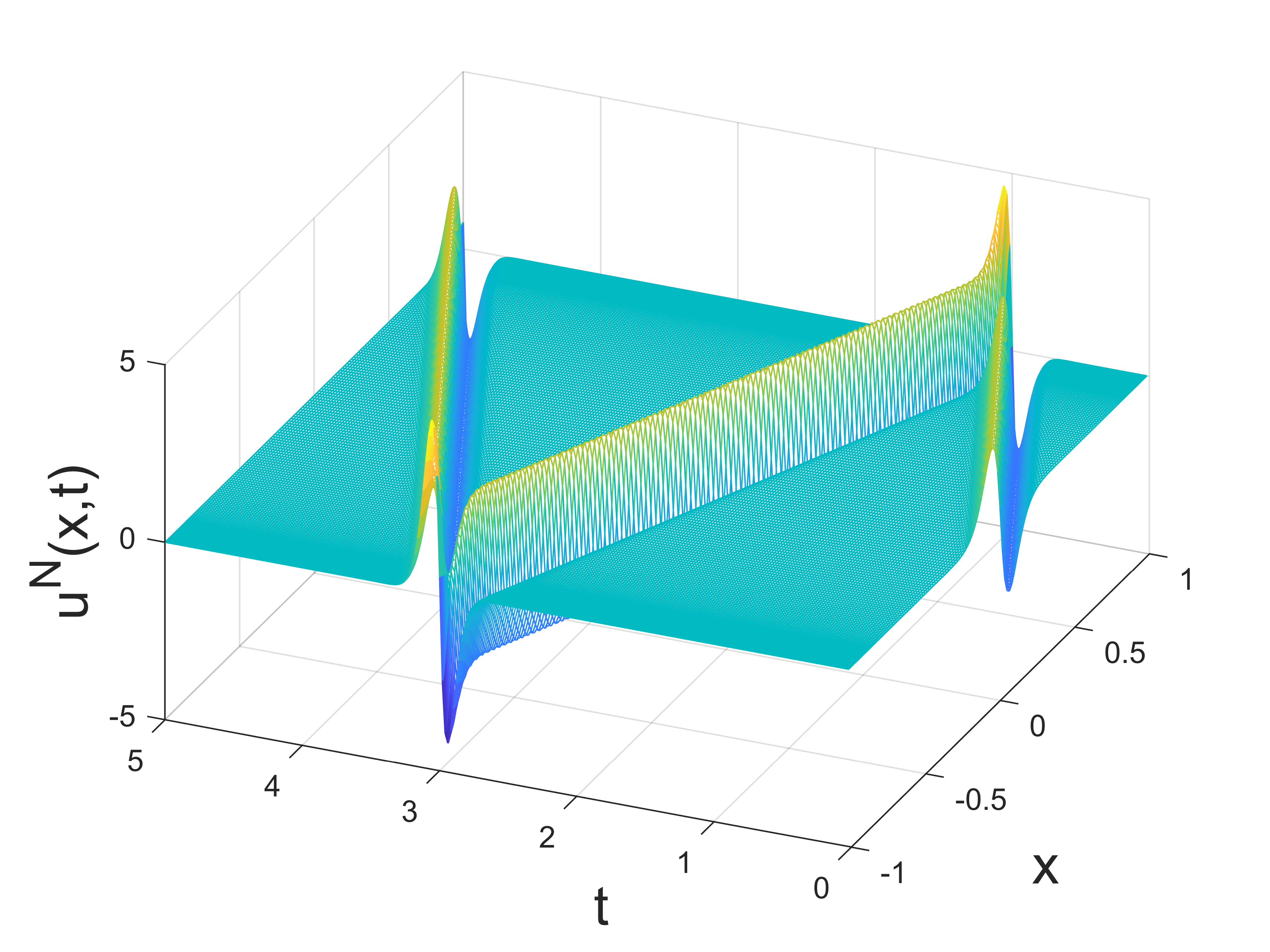}
\caption*{The evolution of a breather soliton in the classical Sine-Gordon equation.}
\end{subfigure}
\caption{With reference to Section~\ref{sec:dispersive}: the comparison between the behavior of a breather-type soliton in the nonlocal model and in the classical Sine-Gordon model. The parameters for the simulation are $\delta=0.2$, $N=200$, $\alpha=0.4$, and $\Delta t=8/N_T$, $N_T=400$.}
\label{fig:breather}
\end{figure}
In all shown cases, due to the dispersive effects, we can also observe a time delay in the soliton waves to reach the boundaries with respect to the local soliton waves.

\subsection{{Test 3: The energy behavior of the fully-discrete scheme}}
\label{sec:energy}

{The aim of the test is to check the behavior of the discrete energy functional associated to the proposed fully-discrete scheme. We take $u_0(x)=e^{-x^2/.002}$, $v_0(x)=0$, as initial conditions and we fix the parameters of the model as $\delta=0.2$, $\alpha=0.4$ and $N=800$ and $\Delta t=10^{-4}$. In Figure~\ref{fig:energy}, we show the behavior of the distribution in time of $E(t)/E(0)$. It is clearly not constant, but oscillates around 1. Due to this, we refer to the proposed model as \textit{almost-energy preserving} in time, in the sense that, although there are oscillations in the behavior of the discretized energy, their amplitude is such as to maintain the oscillations around one with an error of one percent except for the first instants of time. The reasons of such behavior need to be further investigated: it could be related to the fact that even if St\"ormer-Verlet scheme is a symplectic semi-implicit second-order integrator, it does not allow to maintain, in general, the energy-preserving property when applied to nonlinear problems, as shown for instance in~\cite{bilbao2023,hairer2003}.}

\begin{figure}%
\centering
\includegraphics[width=.5\textwidth]{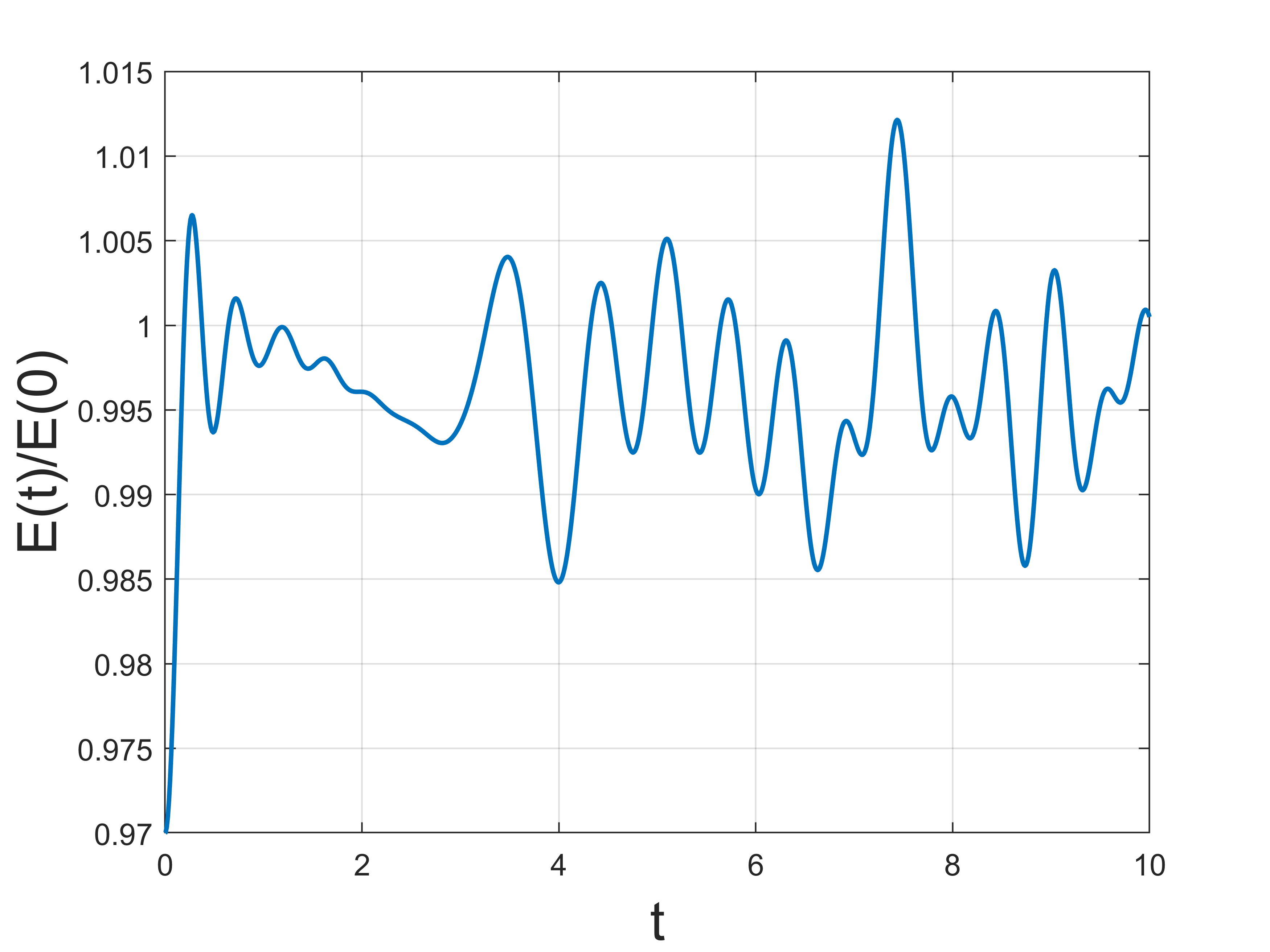}%
\caption{{With reference to Section~\ref{sec:energy}, the behavior of the energy functional}.}%
\label{fig:energy}%
\end{figure}

\section*{Conclusive Remarks} 

A spectral method based on the Chebyshev nodes spatial discretization for the Cauchy problem related to the evolution of the peridynamic Sine-Gordon equation has been studied in this work. The consistency of such discretization has been rigorously proved and various numerical experiments have clarified strength and limitations of the proposed model. Specifically, the proposed model has been validated against a second order centered finite-difference scheme by comparing the evolution in time of several nonlocal soliton-type solutions. {Lastly, dispersive effects as well the behavior of the energy functional of the specific peridynamic kernel have been experimentally showed.}

\section*{Acknowledgments} {The authors thank the anonymous referees for their helpful and constructive remarks that have contributed to significantly improve the original manuscript.} AC, LL, and SFP are members of Gruppo Nazionale per il Calcolo Scientifico (GNCS) of the Istituto Nazionale di Alta Matematica (INdAM).
This work was partially supported by: 
\begin{itemize}
\item Research Project of National Relevance ``Evolution problems involving interacting scales'' granted by the Italian Ministry of Education, University and Research (MUR Prin 2022, project code 2022M9BKBC, Grant No. CUP  D53D23005880006);
\item Research Project of National Relevance ``Mathematical Modeling of Biodiversity in the Mediterranean sea: from bacteria to predators, from meadows to currents'' granted by the Italian Ministry of University and Research (MUR) under the National Recovery and Resilience Plan (NRRP) funded by the European Union - NextGenerationEU (MUR Prin PNRR 2023, project code P202254HT8, Grant No. CUP B53D23027760001);
\item PNRR MUR - M4C2 Project, grant number N00000013 - CUP D93C22000430001;
\item the INdAM-GNCS Project CUP - E53C23001670001.
\item PRIN2022PNRR n. verb+P2022M7JZW+\textit{``SAFER MESH - Sustainable mAnagement oF watEr Resources ModEls and numerical MetHods''} research grant,
funded by the Italian Ministry of Universities and Research (MUR) and  by the European Union through Next Generation EU, M4C2.
\end{itemize}


\end{document}